\documentclass[journal]{new-aiaa}
\usepackage{textcomp}

\usepackage{graphicx}
\usepackage[version=4]{mhchem}
\usepackage{siunitx}
\usepackage{longtable,tabularx}
\usepackage{float}
\usepackage{subfigure}%
\usepackage{bm}

\usepackage{nicefrac}
\usepackage{amsmath}

\usepackage{multirow}
\usepackage{amsthm}

\newtheorem{theorem}{Theorem}
\newtheorem{remark}{Remark}
\newtheorem{corollary}{Corollary}
\newtheorem{lemma}{Lemma}
\usepackage{extarrows}%

\usepackage{cleveref}

\crefname{equation}{}{}
\usepackage{systeme}

\usepackage{relsize}

\usepackage{arydshln,leftidx,mathtools}

\usepackage{booktabs}
\usepackage{siunitx}

\usepackage{stackengine}
\stackMath

\usepackage{breqn}

\usepackage{diagbox}

\usepackage{tablefootnote}
\usepackage{threeparttable}

\usepackage{xcolor}

\newcommand{\RNum}[1]{\uppercase\expandafter{\romannumeral #1\relax}}

\usepackage{enumitem}
\usepackage{makecell}
\usepackage{textcomp}

\usepackage{tablefootnote}
\usepackage{threeparttable}

\usepackage{algorithm}
\usepackage{algpseudocode}
\algnewcommand\And{\textbf{and} }

\usepackage{comment}

\usepackage{tikz}
\usepackage{afterpage}
\usepackage{pgfplots}
\pgfplotsset{compat=1.3}
\usepgfplotslibrary{groupplots}
\pgfkeys{/pgf/number format/.cd,1000 sep={}}

\title{Characterization of Singular Arcs in Spacecraft Trajectory Optimization}

\author{Andrea C. Morelli\footnote{Ph.D. Candidate, Politecnico di Milano, Dept. of Aerospace Science and Technology, Via La Masa, 20156 Milano. Email: andreacarlo.morelli@polimi.it. Corresponding author.}, Carmine Giordano\footnote{PostDoc Fellow, Politecnico di Milano, Dept. of Aerospace Science and Technology, Via La Masa, 20156 Milano. Email: carmine.giordano@polimi.it. AIAA Member.}, Riccardo Bonalli \footnote{Associate Professor, Laboratoire des Signaux et Systèmes, Université Paris-Saclay, CNRS, CentraleSupèlec, B\^at.
Bréguet, 3 Rue Joliot Curie, 91190 Gif-sur-Yvette, France. E-mail: riccardo.bonalli@cnrs.fr}, and Francesco Topputo\footnote{Full Professor, Politecnico di Milano, Dept. of Aerospace Science and Technology, Via La Masa, 20156 Milano. Email: francesco.topputo@polimi.it. AIAA Senior Member.}}

\begin{document}
\maketitle

\begin{abstract}
Low-thrust engines for interplanetary spacecraft transfers allow cost-effective space missions with flexible launch and arrival dates. To find fuel-optimal trajectories, an optimal control problem is to be solved. Pontryagin's Maximum Principle shows that the structure of the optimal control is bang-bang with the possibility of singular arcs. Even though the latter  have been heuristically shown to rarely appear in practical applications, a full theoretical characterization does not exist in the literature. As a growing number of missions are expected to adopt low-thrust engines in the near future, such study is required to have a comprehensive understanding of the problem. This work presents analytical necessary conditions for the existence of singular arcs that only depend on three physical variables and not on the costates. Moreover, it provides an analytical expression of the singular control that depends on a limited set of physical variables. This is a fundamental feature, as simple evaluation of the necessary condition and of the singular control can be performed. Finally, it provides insightful information on the reasons why singular arcs are rare and it quantifies the possibility of their occurrence.
\end{abstract}

\section{Introduction}
\label{sec:introduction}
The recent advancements in the field of low-thrust engines for interplanetary spacecraft transfers have opened completely new mission scenarios. The high efficiency of such thrusters and their continuous thrust enabled cost-effective space missions with flexible launch and arrival dates \cite{topputo2021envelop}. Consequently, an increasing number of missions are expected to adopt this technology in the near future. \\
In space missions, trajectories are usually designed to minimize some key objective, such as the fuel mass consumption \cite{morelli2021robust}. Due to their continuous thrust, low-thrust engines require an Optimal Control Problem (OPC) to be solved to find such trajectories \cite{bryson1969applied}. State-of-the-art methods that are used to solve the problem divide into direct and indirect ones. The former discretize the continuous-time problem and solve the resulting (non)linear program, whereas the latter solve a two-point boundary value problem formulated using the calculus of variation \cite{betts1998survey}. \\ By employing the Pontryagin's Maximum Principle (PMP) \cite{kirk1970optimal} it can be proved that the optimal control of the low-thrust trajectory optimization (LTO) problem follows a bang-bang structure, i.e., the throttle factor should be either at the maximum or the minimum value \cite{topputo2014survey}. However, there could also be cases where the first-order optimality conditions cannot provide any information on the structure of the optimal control. In that case, the control is usually referred to as singular \cite{kirk1970optimal}. A thorough analysis of singular control in the powered descent and landing (PDG) problem was recently performed \cite{leparoux2022structure}. In that case, due to the constant free dynamics, there could only be one singular arc in a given trajectory, and only for specific initial conditions. However, the results are not directly applicable to the  LTO problem given the highly nonlinear free dynamics. In practical applications, if unperturbed two-body dynamics with thrust acceleration are considered, singular thrust arcs are rare and therefore it is not uncommon to solve the LTO problem as if they did not exist \cite{russell2007primer,taheri2016enhanced,bertrand2002new,tang2018fuel,nurre2023duty}, though this is only supported by empirical results rather than theoretical evidence. Previous work has shown that singular arcs can in fact theoretically happen in a two-body dynamical environment and trajectories with several singular arcs can actually be designed \cite{azimov2010extremal}. A further example are the \textit{Lawden’s
spirals} for intermediate thrust arcs \cite{lawden1963optimal}. Others have formulated necessary conditions for the optimality of singular arcs in the case of multiple gravitational bodies \cite{park2013necessary}. So far, researchers have focused on investigating the optimality of singular arcs rather than quantifying, theoretically and numerically, their occurrence. Moreover, in many cases, necessary conditions and expressions of the singular controls were expressed as a function of both states and costates \cite{park2013necessary}, making it impossible to have a physical grasp on the problem. \\ This work proposes for the first time a full characterization of singular arcs in spacecraft LTO in a two-body dynamical environment, but the logic applies to more complex dynamics as well. In reality, more complex dynamics are usually considered when designing a spacecraft trajectory (e.g., $n$-body problem). The bang-bang structure of the optimal control does not depend on the specific dynamics, but the same does not apply to the frequency of singular arcs, which could appear more often. It is therefore relevant to investigate the behaviour of singular arcs to understand whether mission analysts can safely assume that a bang-bang control accurately captures the solution of a LTO problem. The contribution of this article is threefold. First, necessary conditions for having singular arcs that solely depend on three physical variables are provided. Moreover, the singular control is defined through an expression that depends on a limited set of physical variables. Finally, based on our theoretical results, we provide insightful information about the reasons why in practical applications singular arcs rarely appear. In particular, we show that the necessary conditions have a limited number of solutions, given the orbit eccentricity and true anomaly. In addition, we 
show that for typical inner Solar System missions, even in the case the necessary conditions are satisfied, the expression of the singular throttle factor is singular in the minority of the cases. \\ 
The remainder of the paper is organized as follows. Section \ref{sec:problem} formulates the considered LTO problem. Section \ref{sec:strategy} describes the approach that has been used to characterize the thrust arcs, as well as the theoretical findings. Section \ref{sec:genass} presents the numerical simulations. Finally, \ref{sec:conclusion} concludes the work.
\section{Problem Statement}
\label{sec:problem}
The two-body dynamics of a spacecraft around a primary body and equipped with a low-thrust engine can be expressed in Cartesian coordinates as \cite{jiang2012practical}
\begin{equation}
\left\{\begin{split}
&\Dot{x} = v_x\\
&\Dot{y} = v_y\\
&\Dot{z} = v_z\\
&\Dot{v}_x = -\mu \frac{x}{\|\mathbf{r}\|^3} + \frac{T_{\text{max}}}{m}c \cos \delta \cos \gamma \\
&\Dot{v}_y = -\mu \frac{y}{\|\mathbf{r}\|^3} + \frac{T_{\text{max}}}{m}c \cos \delta \sin \gamma \\
&\Dot{v}_z = -\mu \frac{z}{\|\mathbf{r}\|^3} + \frac{T_{\text{max}}}{m}c \sin \delta \\
& \Dot{m} = - \frac{T_{\text{max}}}{I_{\text{sp}}g_0}c
\end{split} \right.
\label{dynamics1}
\end{equation}
$\mathbf{r} = [x, y, z]$ and $\mathbf{v} = [v_x, v_y, v_z]$ are the position and velocity vectors of the spacecraft, respectively, and $m$ is the spacecraft mass. $\mathbf{u} = [c, \delta, \gamma]$ are the controls, where $c \in [0, 1]$ is the thrust throttle factor and $\delta$ and $\gamma$ are the in- and out-of-plane angles of the thrust vector. $T_{\text{max}}$ is the constant maximum thrust of the engine, $I_{\text{sp}}$ is the specific impulse, $g_0$ is the Earth gravity acceleration at sea level, and $\mu$ is the gravitational parameter of the primary body.
Equation \cref{dynamics1} can be written in vectorial form as
\begin{equation}
    \Dot{\mathbf{x}} = \mathbf{f}(\mathbf{x}) + \mathbf{g}(\mathbf{x}, \mathbf{u})
    \label{dyncomp}
\end{equation}
where
\begin{equation}
 \mathbf{f}(\mathbf{x}) =
\begin{bmatrix}
\mathbf{0} \\
        - \mu \frac{\mathbf{r}}{\|\mathbf{r}\|^3} \\
        0
\end{bmatrix}, \quad 
\mathbf{g}(\mathbf{x}, \mathbf{u}) = \begin{bmatrix}
\mathbf{0} \\
        \frac{T_{\text{max}}}{m} c \mathbf{n} \\
        - \frac{T_{\text{max}}}{I_{\text{sp}}g_0}c
\end{bmatrix}
\end{equation}
The vector 
\begin{equation}
    \mathbf{n} = \begin{bmatrix}
        \cos \delta \cos \gamma \\
        \cos \delta \sin \gamma \\
        \sin \delta
    \end{bmatrix}
\end{equation}
represents the thrust direction. In all the above equations, the time dependency has been dropped for brevity. \\ We consider the problem of finding the spacecraft trajectory that minimizes the fuel consumption from a fixed initial boundary condition to a fixed final boundary condition and no path constraints. The objective function can be expressed in Meyer form as \cite{longuski2014optimal}
\begin{equation}
    J = -m(t_f)
\end{equation}
where $t_f$ is the time of flight. Note that minimizing $J$ is equivalent to maximizing the final mass $m(t_f)$. The constraints of the problem are the boundary conditions
\begin{equation}
    \mathbf{x}(t_0) = \mathbf{x}_0, \quad \mathbf{x}(t_f) = \mathbf{x}_f  
    \label{bcs}
\end{equation}
and  the control bounds, namely
\begin{equation}
    0 \leq c \leq 1
    \label{boundsc}
\end{equation}
The optimization problem is formulated as
\begin{equation}
    \underset{\mathbf{u} \in \mathbb{U}}{\text{min}} \, J \quad \text{s.t. Eqs.\ \cref{dyncomp}, \cref{bcs}, and \cref{boundsc}}
    \label{OCP}
\end{equation}
where $\mathbb{U}$ is the set of admissible controls. The Hamiltonian of the system is  \cite{bryson1969applied}
\begin{equation}
        H(\mathbf{x}, \mathbf{u}, \mathbf{p}) = \mathbf{p}_r \cdot \mathbf{v} + \mathbf{p}_v \cdot \left(-\mu \frac{\mathbf{r}}{\|\mathbf{r}\|^3} + \frac{T_{\text{max}}}{m}c \mathbf{n}\right)  -p_m \frac{T_{\text{max}}}{I_{\text{sp}}g_0}c
\end{equation}
where $\mathbf{p} = [\mathbf{p}_r, \mathbf{p}_v, p_m]$ is the vector that collects the position, velocity, and mass costate variables. According to the Pontryagin's Maximum Principle (PMP), the optimal control $\mathbf{u}^*$ maximizes the Hamiltonian \cite{kirk1970optimal}. Therefore, 
\begin{equation}
\begin{split}
    \mathbf{u}^* &= \underset{\mathbf{u} \in \mathbb{U}} {\text{argmax}} \, \left[ (\mathbf{p}_v \cdot \mathbf{n}) \frac{T_{\text{max}}}{m} - p_m \frac{T_{\text{max}}}{I_{\text{sp}}g_0}\right]c 
\end{split}
\label{ustar}
\end{equation}
Let us indicate the optimal thrust direction with $\mathbf{n}^*$, and let us define the switching function
\begin{equation}
    S = (\mathbf{p}_v \cdot \mathbf{n}^*) \frac{T_{\text{max}}}{m} - p_m \frac{T_{\text{max}}}{I_{\text{sp}}g_0}
    \label{Scomplete}
\end{equation}
In turn,
\begin{equation}
\begin{split}
    c^* &= \underset{c \in \mathbb{U}}{\text{argmax}} \, Sc
\end{split}
\end{equation}
 As a consequence,
\begin{equation}
c^* = \left\{\begin{split}
& 1 \hspace{1.3cm} \text{if} \quad S > 0 \\
& 0 \hspace{1.3cm} \text{if} \quad S < 0 \\
& \in (0, 1) \quad \text{if} \quad S = 0
\end{split} \right.
\label{optimalc}
\end{equation}
The third of the cases in Eq.\ \cref{optimalc} represents the singular case. If $\mathbf{p}_v \neq \mathbf{0}$, then $\mathbf{n}^* = \frac{\mathbf{p}_v}{\|\mathbf{p}_v\|}$ and the switching function $S$ becomes
\begin{equation}
     S = \frac{\|\mathbf{p}_v\|}{m} -\frac{p_m}{I_{\text{sp}}g_0}
     \label{Sred}
\end{equation}
Finally, the dynamics of the costates can also be retrieved from the Hamiltonian function and be written as \cite{bryson1969applied}
\begin{equation}
\left\{\begin{split}
&\Dot{\mathbf{p}}_r = -\frac{\partial H}{\partial \mathbf{r}} = -\frac{3\mu}{\|\mathbf{r}\|^5}(\mathbf{r} \cdot \mathbf{p}_v) \mathbf{r} + \frac{\mu}{\|\mathbf{r}\|^3}\mathbf{p}_v\\
&\Dot{\mathbf{p}}_v = -\frac{\partial H}{\partial \mathbf{v}} = -\mathbf{p}_r\\
&\Dot{p}_m = -\frac{\partial H}{\partial m} =  c\frac{\|\mathbf{p}_v\|T_{\text{max}}}{m^2}
\end{split} \right.
\label{costatedynamics}
\end{equation}
\begin{lemma}
Let \(S = 0\) on \(I_s \subset [t_0, t_f]\). Then, $\mathbf{p}_v \neq \mathbf{0}$ and $\mathbf{n}^* = \frac{\mathbf{p}_v}{\|\mathbf{p}_v\|}$ on $I_s$.
\label{lemma1}
\end{lemma}
\begin{proof}
To prove the claim by contradiction, assume that $\mathbf{p}_v = \mathbf{0}$. From Eq.\ \cref{Scomplete}, since $S = 0$, then also $p_m = 0$. Moreover, since $\mathbf{p}_v = \mathbf{0}$, then $\dot{\mathbf{p}}_v = \mathbf{0}$. From the second of Eqs.\ \cref{costatedynamics}, $\mathbf{p}_r = \mathbf{0}$. However, this would violate the nontriviality condition $[ \mathbf{p}_r, \mathbf{p_v}, p_m] \neq \mathbf{0}$ on $I_s$.
\end{proof} 
Therefore, the optimal thrust vector $\mathbf{n}^*$ and the costate $\mathbf{p}_v$ associated with the spacecraft velocity are always parallel on singular arcs. This result will be used throughout the rest of the paper.
\section{General Results}
\label{sec:strategy}
In this section, two of our main findings are presented in the form of theorems, along with several lemmas that will be used to proof the theorems.
\subsection{Main Statements}
\begin{theorem}
Let a LTO problem be described by Eq.\ \cref{OCP}, with the further assumption of planar dynamics. Moreover, let \(S = 0\) on \(I_s\) for that problem. If the angle between the thrust direction  \(\mathbf{n}^*\) and the spacecraft position vector $\mathbf{r}$ is denoted as $\beta$, then the closed-form surface that relates $\beta$ to the eccentricity $e$ and the true anomaly $\theta$ of the spacecraft on $I_s$ is given by 
    \begin{equation}
         \Psi(e,\theta, \beta) = 2 \cos \beta \sin \beta \left( -1 {\color{red} \pm} \sqrt{\frac{1 - 3 \cos^2 \beta}{1 + e \cos \theta}} {\color{blue} \mp} \sqrt{\frac{1 - 3 \cos^2 \beta}{1 + e \cos \theta}} 
 \right) - (1 - 3\cos^2 \beta) \frac{e \sin \theta}{1 + e \cos \theta} = 0
\label{algebraici}
\end{equation}
\label{theorem1}
\end{theorem}
\begin{remark}
As it will be shown later on, the signs ${\color{red} \pm}$ and ${\color{blue} \mp}$ come from different sources and can change independently. Therefore, the two terms do not cancel out a priori.
\end{remark}
Equation \cref{algebraici} represents a set of necessary conditions to have singular arcs. This means that the angle between the optimal thrust direction and the position vector must assume precise values depending on the value of the eccentricity and true anomaly of the spacecraft orbit.
\begin{theorem}
    Let $S = 0$ and $\left(e, \theta, \beta\right)$ satisfy Eq.\ \cref{algebraici} on $I_s$. Moreover, let the term  $A(\beta) = (1 - 3\cos^2 \beta) \cos \beta + 2 \cos \beta \sin^2 \beta \neq 0$ on $I_s$.  Then, 
    \begin{enumerate}
        \item the singular thrust control action $c_s$ on $I_s$ can be expressed as
    \begin{equation}
        c_s = \frac{B(\|\mathbf{r}\|, \beta, e, \theta, m)}{A(\beta)}
    \end{equation}
      \item the term $A(\beta)$ can only be zero  when either $\cos \beta = 0$ or $\sin \beta = \pm \sqrt{\frac{2}{5}}$. The first case corresponds to the conditions $e = 0$ or $\theta = 2k\pi$, $k \in \mathbb{Z}$. Moreover, for that case,
\begin{enumerate}
\item[a.] if $e = 0$, there might exist an interval $\tilde{I}_s \in I_s$ such that $A(\beta) = 0$ on $\tilde{I}_s$ if $\cos \theta = 0$;
\item[b.] if $\theta = 2k\pi$, $k \in \mathbb{Z}$, then $A(\beta) = 0$ only at isolated points, i.e., $\nexists \, \tilde{I}_s \in I_s $ such that $A(\beta) = 0$ on $\tilde{I}_s$.
    \end{enumerate}
    \end{enumerate} 
    \label{theorem2}
\end{theorem}
The terms $B$ and $A$ will be given explicitly in the proof of the theorem. They only depend on the physical variables $\|\mathbf{r}\|, \beta, e, \theta$, and $m$. Therefore, the computation of the singular control $c_s$ can be easily performed in case singular arcs appear while solving a LTO problem. Nonetheless, we will show that the factor $c_s$ is actually singular (i.e, $0 < c_s < 1$) in limited regions of interest of the state space. 
\subsection{Main Lemmas and Corollaries}
\begin{lemma}
    If $S = 0$ on $I_s$, then the costates $\mathbf{p}_v$ an $\mathbf{p}_r$ are perpendicular on $I_s$.
    \label{lemma2}
\end{lemma}
\begin{proof}
    The result from Lemma \ref{lemma1} allows to write the switching function as in Eq.\ \cref{Sred}. Moreover, it must also be
\begin{equation}
    \Dot{S} = -\frac{\|\mathbf{p}_v\|}{m^2}\Dot{m} +  \frac{\mathbf{p}_v \cdot \Dot{\mathbf{p}}_v}{\|\mathbf{p}_v\| m} - \frac{\Dot{p}_m}{I_{\text{sp}}g_0} = 0
\end{equation}
By substituting the expressions for $\Dot{m}$ and $\Dot{p}_m$ contained in Eqs.\ \cref{dynamics1} and \cref{costatedynamics}, it is obtained
\begin{equation}
        \Dot{S} = \frac{\|\mathbf{p}_v\|}{m^2}\frac{T_{\text{max}}}{I_{\text{sp}}g_0}c + \frac{\mathbf{p}_v \cdot \Dot{\mathbf{p}}_v}{\|\mathbf{p}_v\| m} - \frac{\|\mathbf{p}_v\|}{m^2}\frac{T_{\text{max}}}{I_{\text{sp}}g_0}c = \frac{\mathbf{p}_v \cdot \Dot{\mathbf{p}}_v}{\|\mathbf{p}_v\| m}
\end{equation}
Using the second costate equation in Eq.\ \cref{costatedynamics}, one gets
\begin{equation}
    D_1 = \mathbf{p}_v \cdot \Dot{\mathbf{p}}_v = -\mathbf{p}_v \cdot \mathbf{p}_r = 0 
    \label{D1}
\end{equation}
Therefore, the costates $\mathbf{p}_v$ and $\mathbf{p}_r$ are perpendicular along singular arcs.
\end{proof}
\begin{lemma}
Let a LTO problem be described by Eq.\ \cref{OCP}, with the further assumption of planar dynamics. Moreover, let $\alpha$ indicate the angle between the reference direction and the spacecraft radius $\mathbf{r}$. If $S = 0$ on $I_s$, then on $I_s$ it must be
\begin{equation}
     -\|\mathbf{r}\|^3(\Dot{\alpha} + \Dot{\beta})^2 + \mu (1 - 3\cos^2 \beta) = 0
\end{equation}
\label{lemma3}
\end{lemma}
\begin{proof}
Using the result of Lemma \ref{lemma2}, from the hypothesis $S = 0$ follows that $\Dot{D}_1 = 0$ (see Eq.\ \cref{D1}). Consequently,   
\begin{equation}
\begin{split}
        \Dot{D}_1 & = \Dot{\mathbf{p}}_v \cdot \mathbf{p}_r + \mathbf{p}_v \cdot \Dot{\mathbf{p}}_r \\
        & = -\mathbf{p}_r \cdot \mathbf{p}_r + \mathbf{p}_v \cdot \left[  -\frac{3\mu}{\|\mathbf{r}\|^5}(\mathbf{r} \cdot \mathbf{p}_v) \mathbf{r} + \frac{\mu}{\|\mathbf{r}\|^3}\mathbf{p}_v \right] \\
        & = -\|\mathbf{p}_r\|^2 - \frac{3\mu}{\|\mathbf{r}\|^5}(\mathbf{r} \cdot \mathbf{p}_v) (\mathbf{r} \cdot \mathbf{p}_v) + \frac{\mu}{\|\mathbf{r}\|^3}(\mathbf{p}_v \cdot \mathbf{p}_v) = 0\\
        \end{split}
        \label{D1dotfirst}
        \end{equation}
Introducing the hypothesis of planar dynamics, the angles described in Fig.\ \ref{fig:angles} can be defined. Therefore,
\begin{equation}
    \begin{split}
       -\|\mathbf{p}_r\|^2 - \frac{3\mu}{\|\mathbf{r}\|^3} \|\mathbf{p}_v\|^2 \cos^2 \beta + \frac{\mu}{\|\mathbf{r}\|^3} \|\mathbf{p}_v\|^2 = 0
       \end{split}
\end{equation}
From which comes
\begin{equation}
    \begin{split}
     -\|\mathbf{r}\|^3\|\mathbf{p}_r\|^2 + \mu \|\mathbf{p}_v\|^2 (1 - 3\cos^2 \beta) = 0.
\end{split}
\label{D1dot}
\end{equation}
\begin{figure*}[h]
	\centering
	\includegraphics[]{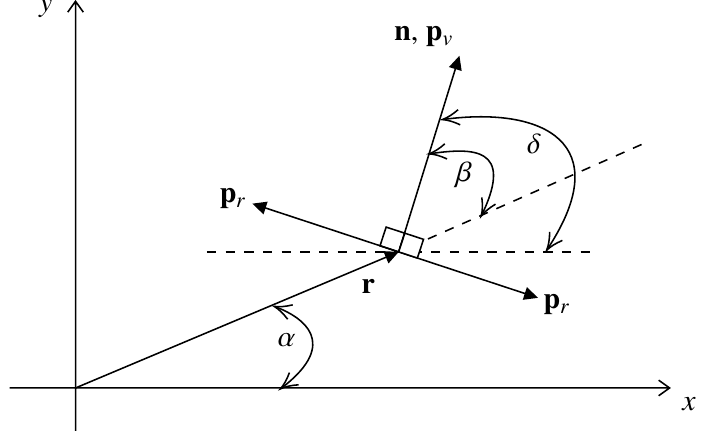}
	\caption{Definition of the angles.}
	\label{fig:angles}
\end{figure*}
Consider the costate $\mathbf{p}_v = \|\mathbf{p}_v\| \mathbf{q}$, where $\|\mathbf{q}\| = 1$. Deriving the vector with respect to time yields
\begin{equation}
   \Dot{\mathbf{p}}_v = \frac{\text{d}}{\text{d}t}\|\mathbf{p}_v\|\mathbf{q} + \|\mathbf{p}_v\| \Dot{\mathbf{q}} =  \frac{\text{d}}{\text{d}t}\|\mathbf{p}_v\| \mathbf{q} + \|\mathbf{p}_v\| \Dot{\delta} \mathbf{s},
\end{equation}
where $\mathbf{s}$ is a unitary vector perpendicular to $\mathbf{q}$. The dynamics of the velocity costate in Eqs.\ \cref{costatedynamics} and Eq.\ \cref{D1} show that the component of $ \Dot{\mathbf{p}}_v$ along $\mathbf{q}$ must be zero. In turn,
\begin{equation}
    \Dot{\mathbf{p}}_v = -\mathbf{p}_r = \|\mathbf{p}_v\| \Dot{\delta} \mathbf{s}
\end{equation}
and
\begin{equation}
    \|\Dot{\mathbf{p}}_v\| = \|\mathbf{p}_r\| = \|\mathbf{p}_v\| |\Dot{\delta}|
    \label{pr}
\end{equation}
From Fig.\ \ref{fig:angles}, it can be observed that the angle $\delta = \alpha + \beta$, so
\begin{equation}
   \|\mathbf{p}_r\| = \|\mathbf{p}_v\| |\Dot{\alpha} + \Dot{\beta}|
\end{equation}
Finally, Eq.\ \cref{D1dot} becomes
\begin{equation}
    \begin{split} 
         0 & = -\|\mathbf{r}\|^3 \|\mathbf{p}_v\|^2 (\Dot{\alpha} + \Dot{\beta})^2 + \mu \|\mathbf{p}_v\|^2 (1 - 3\cos^2 \beta)\\ 
         & = -\|\mathbf{r}\|^3(\Dot{\alpha} + \Dot{\beta})^2 + \mu (1 - 3\cos^2 \beta) = D_2
    \end{split}
    \label{D2}
\end{equation}
\end{proof}
\begin{corollary}
In order to have singular arcs, it must be verified
        \begin{equation}
        1 - 3\cos^2 \beta \geq 0
    \end{equation}
    \label{corollary1}
\end{corollary}
\begin{proof}
    The proof directly follows from Lemma \ref{lemma3}.
\end{proof}
\begin{remark}
    The condition in Corollary \ref{corollary1} corresponds to the following values of $\beta$:
\begin{equation}
\begin{split}
    I_\beta = I_\beta^1 \cup I_\beta^2 
\end{split}
\end{equation}
where
\begin{equation}
\begin{split}
    I_\beta^1 &= \left\{ \beta : \beta_0 + 2k \pi \leq \beta \leq \pi - \beta_0 + 2k \pi, k \in \mathbb{N} \right\} \\
    I_\beta^2 &= \left\{\beta_0 + (2k+1) \pi \leq \beta \leq 2\pi(k+1) - \beta_0, k \in \mathbb{N} \right\} 
\end{split}
\end{equation}
with $\beta_0 = \arccos{\sqrt{\frac{1}{3}}}$.
\label{thetavalues}
\end{remark}

\subsection{Proof of Theorem 1}
\begin{proof}
     In order to have singular arcs, it must be $\Ddot{D}_1 = 0$ (see Eq.\ \cref{D1dotfirst}). Considering Eq.\ \cref{D1dot},
\begin{equation}
\begin{split}
        -2\|\mathbf{r}\|^3 (\mathbf{p}_r \cdot \Dot{\mathbf{p}}_r) -  3 \|\mathbf{r}\|\|\mathbf{p}_r\|^2 (\mathbf{r} \cdot \Dot{\mathbf{r}}) + 2\mu(1 - 3\cos^2 \beta) (\mathbf{p}_v \cdot \Dot{\mathbf{p}}_v) + 6\mu \|\mathbf{p}_v\|^2 \cos \beta \sin \beta \Dot{\beta} = 0
\end{split}
\label{d2dot_costates}
\end{equation}
Using the result from Lemma \ref{lemma2}, the equation becomes
\begin{equation}
    \begin{split}
        D_3 = -2\|\mathbf{r}\|^3 (\mathbf{p}_r \cdot \Dot{\mathbf{p}}_r) -  3 \|\mathbf{r}\|\|\mathbf{p}_r\|^2 (\mathbf{r} \cdot \Dot{\mathbf{r}}) + 6\mu \|\mathbf{p}_v\|^2 \cos \beta \sin \beta \Dot{\beta} = 0
    \end{split}
    \label{D1ddot}
\end{equation}
Making use of the costate equation for $\mathbf{p}_r$, the term $(\mathbf{p}_r \cdot \Dot{\mathbf{p}}_r)$ can be expressed as
\begin{equation}
\begin{split}
        (\mathbf{p}_r \cdot \Dot{\mathbf{p}}_r) &= \mathbf{p}_r \cdot \left(-\frac{3\mu}{\|\mathbf{r}\|^5}(\mathbf{r} \cdot \mathbf{p}_v) \mathbf{r} + \frac{\mu}{\|\mathbf{r}\|^3}\mathbf{p}_v\right) \\
        &=  -\frac{3\mu}{\|\mathbf{r}\|^5}(\mathbf{r} \cdot \mathbf{p}_v) (\mathbf{r} \cdot \mathbf{p}_r)
        \\
        &= -\frac{3\mu}{\|\mathbf{r}\|^3}\|\mathbf{p}_v\| \|\mathbf{p}_r\| \cos \beta \cos \left(\beta {\color{blue}\pm} \frac{\pi}{2}\right) 
\end{split}
\end{equation}
In the above equation, Lemma \ref{lemma2} has been used to write $(\mathbf{r} \cdot \mathbf{p}_r)$  as $\|\mathbf{r}\| \|\mathbf{p}_r\| \cos\left(\beta {\color{blue}\pm} \frac{\pi}{2}\right)$. Note that as also shown in Fig.\ \ref{fig:angles}, the orientation of $\mathbf{p}_r$ is not known, and therefore both the signs $\pm$ need to be considered. Since $\cos\left(\beta {\color{blue}\pm} \frac{\pi}{2}\right)= {\color{blue}\mp} \sin \beta $, \begin{equation}
\begin{split}
         (\mathbf{p}_r \cdot \Dot{\mathbf{p}}_r) &= {\color{blue}\pm} \frac{3\mu}{\|\mathbf{r}\|^3}\|\mathbf{p}_v\| \|\mathbf{p}_r\| \cos \beta \sin \beta \\
         & = {\color{blue}\pm} \frac{3\mu}{\|\mathbf{r}\|^3}\|\mathbf{p}_v\|^2 |\Dot{\delta}| \cos \beta \sin \beta  \quad \leftarrow \text{from Eq.\ \cref{pr}}\\
         & = {\color{blue}\pm} \frac{3\mu}{\|\mathbf{r}\|^3}\|\mathbf{p}_v\|^2 \cos \beta \sin \beta |\Dot{\alpha} + \Dot{\beta}| \quad \leftarrow \text{from Fig.\ \ref{fig:angles}}
\end{split}
\label{prprdot}
\end{equation}
By substituting this expression back in Eq.\ \cref{D1ddot}:
\begin{equation}
\begin{split}
         & {\color{blue}\mp} 6\mu \|\mathbf{p}_v\|^2 \cos \beta \sin \beta| \Dot{\alpha} + \Dot{\beta}| -  3 \|\mathbf{r}\| \|\mathbf{p}_r\|^2 (\mathbf{r} \cdot \Dot{\mathbf{r}}) + 6\mu \|\mathbf{p}_v\|^2 \cos \beta \sin \beta \Dot{\beta} = 0\\
         &  {\color{blue}\mp} 6\mu \|\mathbf{p}_v\|^2 \cos \beta \sin \beta |\Dot{\alpha} + \Dot{\beta}|-  3 \|\mathbf{r}\| \|\mathbf{p}_v\|^2 (\Dot{\alpha} + \Dot{\beta})^2 (\mathbf{r} \cdot \Dot{\mathbf{r}}) + 6\mu \|\mathbf{p}_v\|^2 \cos \beta \sin \beta \Dot{\beta} = 0\\
        &   {\color{blue}\mp}2 \mu \cos \beta \sin \beta |\Dot{\alpha} + \Dot{\beta}|-  \|\mathbf{r}\| (\Dot{\alpha} + \Dot{\beta})^2 (\mathbf{r} \cdot \Dot{\mathbf{r}}) + 2\mu \cos \beta \sin \beta \Dot{\beta} = 0
\end{split}
\label{eq:3rdinter}
\end{equation}
From Lemma \ref{lemma3}, 
\begin{equation}
    (\Dot{\alpha} + \Dot{\beta})^2 = \frac{\mu}{\|\mathbf{r}\|^3} (1 - 3\cos^2 \beta)
    \label{alphaDbetaD_squared}
\end{equation}
\begin{equation}
    |\Dot{\alpha} + \Dot{\beta}| = \sqrt{\frac{\mu}{\|\mathbf{r}\|^3} (1 - 3\cos^2 \beta)}
    \label{alphaDbetaD_mod}
\end{equation}
\begin{equation}
    \Dot{\beta} = -\Dot{\alpha} {\color{red}\pm} \sqrt{\frac{\mu}{\|\mathbf{r}\|^3} (1 - 3\cos^2 \beta)}
\end{equation}
Now, we want to express $\Dot{\alpha}$ as a function of physical variables only, such as orbital parameters. In case of planar transfers, the angle $\alpha$ can be expressed as \cite{curtis2013orbital}
\begin{equation}
    \alpha = \theta + \omega,
    \label{alphathetaomega}
\end{equation}
where $\omega$ is the argument of periapsis, which is defined as \cite{curtis2013orbital}
\begin{equation}
   \omega = \arccos{\frac{\mathbf{N} \cdot \mathbf{e}}{\|\mathbf{N}\| \|\mathbf{e}\|}}
    \label{omegadef}
\end{equation}
$\mathbf{N}$ is a vector pointing towards the ascending node of the orbit, and $\mathbf{e}$ is the eccentricity vector. In case of a planar orbit, $\mathbf{N}$ is undefined. By convention, it is assumed that it coincides with the reference direction $x$ and therefore the angle $\omega$ is defined as the angle between the reference direction and the eccentricity vector, hence Eq.\ \cref{omegadef}. By deriving Eq.\ \cref{alphathetaomega}:
\begin{equation}
\begin{split}
        \Dot{\alpha} & = \Dot{\theta} + \Dot{\omega} = \Dot{\theta}_{\text{2B}} + \Dot{\theta}_{\text{P}} + \Dot{\omega}_{\text{2B}} + \Dot{\omega}_{\text{P}} \end{split}
        \end{equation}
where the terms related to the two-body motion and to orbital perturbations have been highlighted and identified with $(\cdot)_{\text{2B}}$ and $(\cdot)_{\text{P}}$, respectively. In case of planar orbits, it can be proved that $\Dot{\omega}_{\text{P}} = - \Dot{\theta}_{\text{P}}$ \cite{curtis2013orbital}. Moreover, since in two-body motion the argument of periapsis does not change, $\Dot{\omega}_{\text{2B}} = 0$. Therefore,
\begin{equation}
    \begin{split}
        \dot{\alpha}
        = \Dot{\theta}_{\text{2B}} = \frac{\|\mathbf{h}\|}{\|\mathbf{r}\|^2} = \sqrt{\frac{\mu}{\|\mathbf{r}\|^3}(1 + e \cos \theta)}
\end{split}
\label{alphadot}
\end{equation}
where 
\begin{equation}
    \|\mathbf{h}\| = \sqrt{\mu \|\mathbf{r}\|(1 + e \cos \theta)}
    \label{h}
\end{equation}
is the norm of the orbital specific angular momentum \cite{curtis2013orbital}.
Therefore, one can write $\Dot{\beta}$ as
\begin{equation}
    \Dot{\beta} = \sqrt{\frac{\mu}{\|\mathbf{r}\|^3}(1 + e \cos \theta)} \left(-1 {\color{red}\pm} \sqrt{\frac{1 - 3 \cos^2 \beta}{1 + e \cos \theta}} \right) 
    \label{betaD}
\end{equation}
Moreover, we have that
\begin{equation}
    (\mathbf{r} \cdot \Dot{\mathbf{r}}) = (\mathbf{r} \cdot \mathbf{v}) = \|\mathbf{r}\| v_r = \|\mathbf{r}\| \frac{\mu}{\|\mathbf{h}\|} e \sin \theta = e \sin \theta\sqrt{\frac{\mu \|\mathbf{r}\|}{1 + e \cos \theta}},
    \label{rv}
\end{equation}
where 
\begin{equation}
    v_r = \frac{\mu}{\|\mathbf{h}\|} e \sin \theta
    \label{vr}
\end{equation} 
is the component of the velocity along the radius direction \cite{curtis2013orbital}.
By substituting Eqs.\ \cref{alphaDbetaD_squared,alphaDbetaD_mod,betaD,rv} in Eq.\ \cref{eq:3rdinter} we obtain
\begin{equation}
\begin{split}
         \Psi(e,\theta, \beta) = 2 \cos \beta \sin \beta \left( -1 {\color{red} \pm} \sqrt{\frac{1 - 3 \cos^2 \beta}{1 + e \cos \theta}} {\color{blue} \mp} \sqrt{\frac{1 - 3 \cos^2 \beta}{1 + e \cos \theta}} 
 \right) - (1 - 3\cos^2 \beta) \frac{e \sin \theta}{1 + e \cos \theta} = 0
\end{split}
\label{algebraic}
\end{equation}
\end{proof}
Figure \ref{fig:constr} shows $\Psi$ as a function of $\beta$ for $e = 0.2$ and $\theta = \frac{5}{12}\pi$ for the different combinations of the signs.
\begin{remark}
    It can be observed that:
\begin{enumerate}
    \item the function is only defined when $1 - 3\cos^2 \beta \geq 0$, and therefore in the two sub-domains of $[0, 2\pi]$ defined in Remark \ref{thetavalues};
    \item since $\cos (\beta + \pi) = -\cos \beta$ and $\sin (\beta + \pi) = -\sin \beta$, the function has the same behaviour for both the sub-domains.
    \item when $\beta = \beta_0$, the function $\Psi = 2 (\cos^2 \beta_0 - \sin^2 \beta_0) < 0$ and when $\beta = \pi - \beta_0$ then $\Psi = 2 [\cos^2 (\pi - \beta_0) - \sin^2 (\pi - \beta_0)] > 0$.
\end{enumerate}
\label{remark3}
\end{remark}
\begin{corollary}
    For a fixed couple $\left(\Bar{e}, \Bar{\theta}\right)$, the equation $\Psi\left(\Bar{e}, \Bar{\theta}, \beta\right) = 0$ defined in Theorem \ref{theorem1} has at least 6 zeros and at most 10.
    \label{corollary2}
\end{corollary}
\begin{proof}
The function $\Psi\left(\Bar{e}, \Bar{\theta}, \beta\right) = 0$ should be studied in the interval $I_\beta$. However, due to the second point of Remark \ref{remark3}, it can be studied inside the interval $I_\beta^1$ and the same results apply for $I_\beta^2$.
Let us compute the derivative of the function $\Psi(\Bar{e}, \Bar{\theta}, \beta)$ with respect to $\beta$:
\begin{equation}
\begin{split}
        \frac{\text{d}\Psi}{\text{d} \beta}(\Bar{e}, \Bar{\theta}, \beta) &= 2(\cos^2 \beta - \sin^2 \beta) \left( -1 {\color{red} \pm} \sqrt{\frac{1 - 3 \cos^2 \beta}{1 + \Bar{e} \cos \Bar{\theta}}} {\color{blue} \mp} \sqrt{\frac{1 - 3 \cos^2 \beta}{1 + \Bar{e} \cos \Bar{\theta}}}\right)+6\frac{ \cos^2 \beta \sin^2 \beta}{\sqrt{1 - 3\cos^2 \beta}} \left({\color{red} \pm} \frac{1}{\sqrt{1 + \Bar{e} \cos \Bar{\theta}}} {\color{blue} \mp} \frac{1}{\sqrt{1 + \Bar{e} \cos \Bar{\theta}}} \right)  \\
        & -6\cos \beta \sin \beta \frac{e \sin \Bar{\theta}}{1 + \Bar{e} \cos \Bar{\theta}}
\end{split}
\end{equation}
Now, depending on the signs inside the function, we have three cases, correspondent to the three cases in Fig.\ \ref{fig:constr}. 
\subsubsection{Opposite Signs}
Let us first consider the case in which the two signs are opposite. We get:
\begin{equation}
\begin{split}
         \frac{\text{d}\Psi_2}{\text{d} \beta} &= -2(\cos^2 \beta - \sin^2 \beta) - 6\cos \beta \sin \beta \frac{e \sin \Bar{\theta}}{1 + \Bar{e} \cos \Bar{\theta}} = -2(\cos^2 \beta - \sin^2 \beta) - 6\cos \beta \sin \beta \Gamma, \quad \Gamma = \text{const.}
\end{split}
\end{equation}
Let us analyze the first term of the function, i.e., $T_1 = -2(\cos^2 \beta - \sin^2 \beta)$. Inside $I_\beta^1$, it is always positive, has a maximum in $\beta = \pi/2$, and it is symmetric with respect to the $\beta = \pi/2$ axis. The term $T_2 = - 6\cos \beta \sin \beta \Gamma$ inside the same interval is instead monotonic and
\begin{itemize}
    \item[1.] negative in $ I_\beta^{(1,1)} = \left\{\beta_0 \leq \beta < \pi/2\right\}$ and positive in the interval $ I_\beta^{(1,2)} = \left\{ \pi/2 <  \beta \leq \pi - \beta_0\right\}$ if $\Gamma > 0$;
    \item[2.] positive in $I_\beta^{(1,1)} = \left\{\beta_0 \leq \beta < \pi/2\right\}$ and negative in the interval $ I_\beta^{(1,2)} = \left\{ \pi/2 <  \beta \leq \pi - \beta_0\right\}$ if $\Gamma < 0$;
    \item[3.] constantly zero if $\Gamma = 0$.
\end{itemize}  
\paragraph{Case $\Gamma > 0$} If $|T_2(\beta_0)| > |T_1(\beta_0)|$, the derivative $\frac{\text{d}\Psi_2}{\text{d} \beta}$ is negative at $\beta_0$. As $T_2$ increases, the derivative crosses zero and becomes positive. Since $T_1 > 0$ and $T_2$ is monotonic and using the third point of Remark \ref{remark3}, the function has exactly one root. If $|T_2(\beta_0)| < |T_1(\beta_0)|$ the derivative is always positive. Therefore, the function $\Psi$ only has one root.
\paragraph{Case $\Gamma < 0$} In this case, the derivative of $\Psi$ is always positive in $I_\beta^{(1,1)}$. The derivative of $\Psi$ can change sign at most once in $I_\beta^{(1,2)}$ because $T_2$ is monotonic and $T_1 > 0$ and, using the third point of Remark \ref{remark3}, $\Psi$ at $\beta = \pi - \beta_0$ is positive. It follows that the function has only one root.
\paragraph{Case $\Gamma = 0$} In this case, $T_2 \equiv 0$ and the derivative of $\Psi$ is always positive. It follows that $\Psi$ has one root only.
\subsubsection{Both Signs Negative}
Consider now the case when both signs are negative. The derivative of $\Psi$ takes the form
\begin{equation}
\begin{split}
        \frac{\text{d}\Psi_3}{\text{d} \beta} &= -2(\cos^2 \beta - \sin^2 \beta) \left(1 +2\sqrt{\frac{1 - 3 \cos^2 \beta}{1 + \Bar{e} \cos \Bar{\theta}}} \right)- 12 \frac{\cos^2 \beta \sin^2 \beta}{\sqrt{1 - 3\cos^2 \beta}\sqrt{1 + \Bar{e} \cos \Bar{\theta}}} -6\cos \beta \sin \beta \Gamma = T_1 + T_2 + T_3
\end{split}
\end{equation}
First note that when $\beta \rightarrow \beta_0$ and $\beta \rightarrow \pi - \beta_0$, then $\frac{\text{d}\Psi_3}{\text{d} \beta} \rightarrow -\infty$. Let us consider two values $\beta_1, \beta_2$ such that
\begin{enumerate}
    \item $ \frac{\text{d}\Psi_3}{\text{d} \beta} < 0$ in $(\beta_0, \beta_1)$ and $\frac{\text{d}\Psi_3}{\text{d} \beta} = 0$ at $\beta = \beta_1 < \pi/2$;
    \item $ \frac{\text{d}\Psi_3}{\text{d} \beta} < 0$ in $(\beta_2, \pi - \beta_0)$ and $\frac{\text{d}\Psi_3}{\text{d} \beta} = 0$ at $\beta = \beta_2 > \pi/2$;
    \item $\frac{\text{d}\Psi_3}{\text{d} \beta} $ can change sign in $(\beta_1, \beta_2)$.
\end{enumerate}
We want to show that $\frac{\text{d}\Psi_3}{\text{d} \beta} > 0$ always in $(\beta_1, \beta_2)$, from which would follow that $\Psi_3$ has only one root. The following statements hold in the interval $I_\beta^1$:
\begin{enumerate}
    \item $T_1(\beta) = T_1(\pi - \beta)$, $T_2(\beta) = T_2(\pi - \beta)$, and $T_3(\beta) = - T_3(\pi - \beta)$;
    \item $T_1$ is always positive and has a maximum at $\beta = \pi/2$;
    \item $T_2$ is always negative and it is zero at $\beta = \pi/2$.
    \item $T_3$ has already been analyzed before.
\end{enumerate}
\paragraph{Case $\Gamma > 0$} Let us consider the interval $(\beta_1, \pi/2]$. Inside the interval, all the terms $T_i$, $i = 1, \dots, 3$ are growing monotonically. Since $\frac{\text{d}\Psi_3}{\text{d} \beta}(\beta_1) = 0$, then $\frac{\text{d}\Psi_3}{\text{d} \beta}$ is always positive inside the considered interval. Due to the symmetries of $T_1$, $T_2$, and $T_3$, the derivative of $\Psi$ in the interval $[\pi/2, \beta_2]$ decreases monotonically until it becomes zero at $\beta_2$. The function $\frac{\text{d}\Psi_3}{\text{d} \beta}$ is therefore always positive in $(\beta_1, \beta_2)$. Consequently, $\Psi$ has only one zero.
\paragraph{Case $\Gamma < 0$} In $[\beta_1, \pi/2]$, $\frac{\text{d}\Psi_3}{\text{d} \beta} > 0$ because $T_1 >0$, $T_3 > 0$, and $T_2$ increases monotonically. In $[\pi/2, \beta_2]$, $T_2$ and $T_3$ are negative but also monotonic. This means that $\frac{\text{d}\Psi_3}{\text{d} \beta}$ decreases monotonically until it becomes zero at $\beta_2$. The function $\frac{\text{d}\Psi_3}{\text{d} \beta}$ is therefore always positive in $(\beta_1, \beta_2)$. Consequently, $\Psi$ has only one zero.
\paragraph{Case $\Gamma = 0$}  In $[\beta_1, \pi/2]$, $\frac{\text{d}\Psi_3}{\text{d} \beta} > 0$ because $T_1 > 0$ and $T_2$ increases monotonically.  In $[\pi/2, \beta_2]$, $T_2$ is negative but also monotonic. This means that $\frac{\text{d}\Psi_3}{\text{d} \beta}$ decreases monotonically until it becomes zero at $\beta_2$. The function $\frac{\text{d}\Psi_3}{\text{d} \beta}$ is therefore always positive in $(\beta_1, \beta_2)$. Consequently, $\Psi$ has only one zero.
\subsubsection{Both Signs Positive}
Finally consider the case when both of the signs inside $\Psi$ are positive. The derivative becomes:
\begin{equation}
\begin{split}
        \frac{\text{d}\Psi_1}{\text{d} \beta} &= -2(\cos^2 \beta - \sin^2 \beta) \left(1 -2\sqrt{\frac{1 - 3 \cos^2 \beta}{1 + \Bar{e} \cos \Bar{\theta}}} \right)+ 12 \frac{\cos^2 \beta \sin^2 \beta}{\sqrt{1 - 3\cos^2 \beta}\sqrt{1 + \Bar{e} \cos \Bar{\theta}}} -6\cos \beta \sin \beta \Gamma = T_1 + T_2 + T_3
\end{split}
\end{equation}
First note that when $\beta \rightarrow \beta_0$ and $\beta \rightarrow \pi - \beta_0$, then $\Psi_1 \rightarrow +\infty$. Let us define two angles $\beta_1, \beta_2$ such that
\begin{enumerate}
    \item $ \frac{\text{d}\Psi_1}{\text{d} \beta} > 0$ in $(\beta_0, \beta_1)$ and $\frac{\text{d}\Psi_3}{\text{d} \beta} = 0$ at $\beta = \beta_1 < \pi/2$;
    \item $\frac{\text{d}\Psi_1}{\text{d} \beta} > 0$ in $(\beta_2, \pi - \beta_0)$ and $\frac{\text{d}\Psi_1}{\text{d} \beta} = 0$ at $\beta = \beta_2 > \pi/2$;
    \item $\frac{\text{d}\Psi_1}{\text{d} \beta} $ can change sign in $(\beta_1, \beta_2)$.
\end{enumerate}
The idea is to show that $\frac{\text{d}\Psi_1}{\text{d} \beta} < 0$ inside $(\beta_1, \beta_2)$. If this happens, then $\Psi$ has \textit{at most} three zeros and \textit{at least} one.  The following statements hold in the interval $I_\beta^1$:
\begin{enumerate}
    \item $T_1(\beta) = T_1(\pi - \beta)$, $T_2(\beta) = T_2(\pi - \beta)$, and $T_3(\beta) = - T_3(\pi - \beta)$;
    \item $T_1$ has a negative minimum at $\beta = \pi/2$;
    \item $T_2$ is always positive and it is zero at $\beta = \pi/2$.
    \item $T_3$ has already been analyzed before.
\end{enumerate}
Regardless of the sign of the term $T_3$, the angle $\beta_1 < \pi/2$ due to Properties 2), 3), and 4) of the above item list.
\paragraph{\textbf{Case} $\Gamma > 0$} Due to the monotonicity of the functions $T_1$ and $T_2$ in the interval $(\beta_1, \pi/2)$ and since $T_3 < 0$ inside the same interval, we have that $\frac{\text{d}\Psi_1}{\text{d} \beta} \leq 0$ in the interval $[\beta_1, \pi/2)$. In $\beta = \pi/2$, $T_1 < 0$, $T_2 = T_3 = 0$, therefore the derivative is negative. Since $T_1$, $T_2$, and $T_3$ are monotonic,  $\frac{\text{d}\Psi_1}{\text{d} \beta}$ is negative and increasing in the interval $[\pi/2, \beta_2)$. Consequently, $\Psi$ can have maximum $3$ zeros and minimum $1$, depending on the values of $T_1$, $T_2$, and $T_3$ at $\beta = \beta_1$ and $\beta = \beta_2$. \\
\paragraph{\textbf{Case} $\Gamma < 0$} Due to the monotonicity of the functions $T_1$, $T_2$, and $T_3$ in the interval $(\beta_1, \pi/2)$, we have that $\frac{\text{d}\Psi_1}{\text{d} \beta}$ is negative and increasing in the interval $[\beta_1, \pi/2)$. In $\beta = \pi/2$, $T_1 < 0$, $T_2 = T_3 = 0$, therefore the derivative is negative. Since $T_1$, $T_2$ are monotonic and $T_3 < 0$ in the interval $[\pi/2, \beta_2)$, then $\frac{\text{d}\Psi_1}{\text{d} \beta}$ is negative and increasing. Consequently, $\Psi$ can have maximum $3$ zeros and minimum $1$, depending on the values of $T_1$, $T_2$, and $T_3$ at $\beta = \beta_1$ and $\beta = \beta_2$. \\
\paragraph{\textbf{Case} $\Gamma = 0$} The same conclusion of the first two cases directly follows. \\
Therefore, the function $\Psi$ can have at most 5 zeros and at least 3 inside interval $I_\beta^1$. The same reasoning applies to the interval $I_\beta^2$, hence the total number of zeros of the function $\Psi$ is at most 10 and at least 6.
\begin{figure}[!ht]
    \centering
    \subfigure[][{Algebraic constraints with both signs positive.}]
    {\includegraphics[width=0.49 \textwidth]{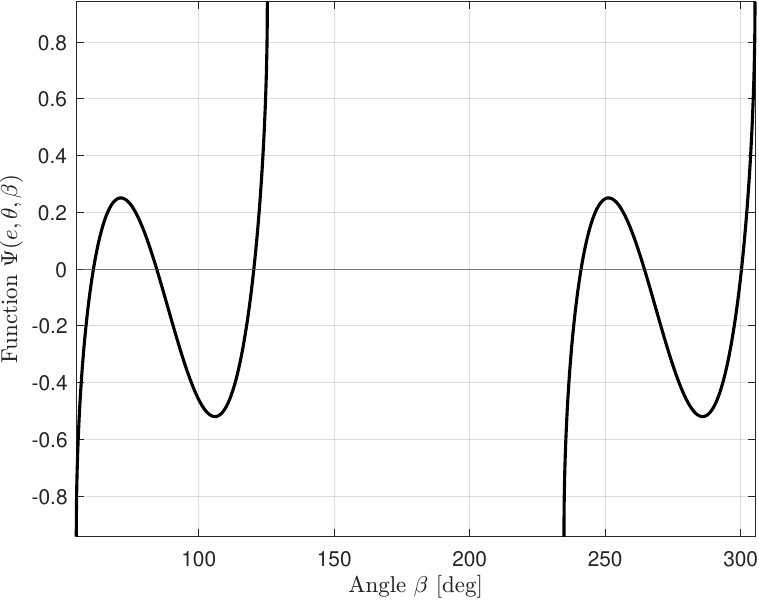}}
    \subfigure[][{Algebraic constraints with one sign positive and one sign negative.}]
    {\includegraphics[width=0.49 \textwidth]{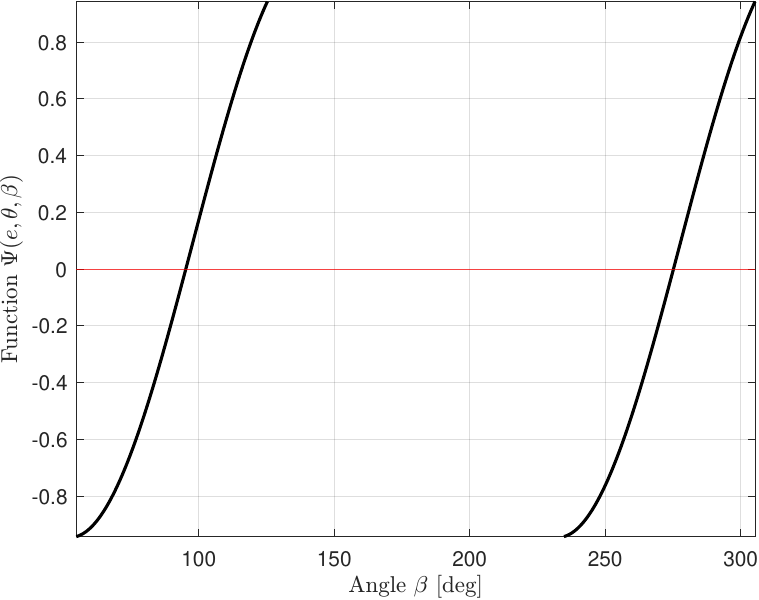}}\\
    \subfigure[][{Algebraic constraints with both signs negative.}]
    {\includegraphics[width=0.49 \textwidth]{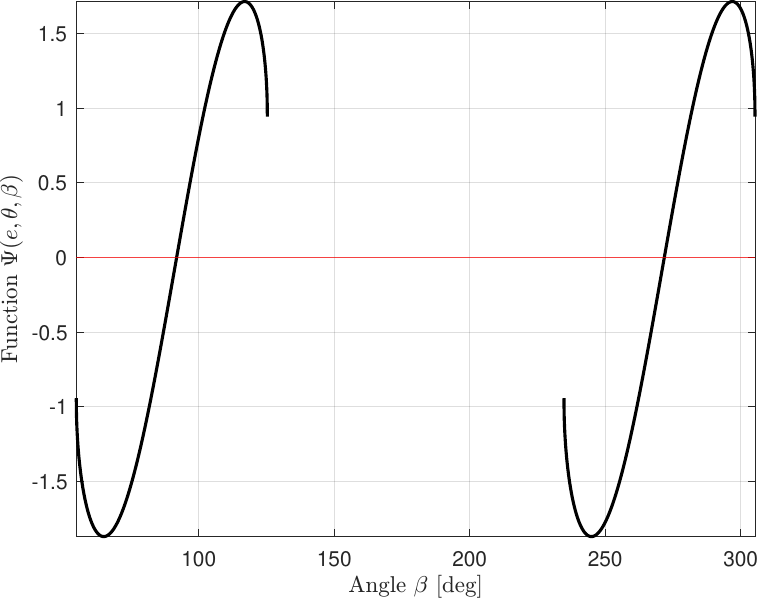}}
    \caption{Algebraic constraint function for a couple of values of $e$ and $\theta$.}
\label{fig:constr}
\end{figure}
\end{proof}
Corollary \ref{corollary2} provides insightful information about the rareness of singular arcs: for a given couple $(\bar{e}, \bar{\theta})$, there only is a discrete set of angles $\beta$ that satisfy the necessary condition. Therefore, the thruster must assume a specific direction, otherwise no singular arcs can happen. Not only this must be verified, but the tuple $(e, \theta, \beta)$ must move on a surface.
\subsection{Proof of Theorem 2.1}
\begin{proof}
Let us consider Eq.\ \cref{D1ddot}. Using the expression for $(\mathbf{p}_r \cdot \Dot{\mathbf{p}}_r)$ from Eq.\ \cref{prprdot}, it becomes
\begin{equation}
   D_3 = {\color{blue} \mp}2\mu\|\mathbf{p}_v\| \|\mathbf{p}_r\| \cos \beta \sin \beta - \|\mathbf{r}\|\|\mathbf{p}_r\|^2(\mathbf{r} \cdot \mathbf{v}) + 2\mu \|\mathbf{p}_v\|^2 \cos \beta \sin \beta \Dot{\beta} = 0
\end{equation}
By deriving the above equation, one obtains
\begin{equation}
\begin{split}
        \dot{D}_3 &= {\color{blue} \mp}2\mu \left[ \frac{(\mathbf{p}_v \cdot \Dot{\mathbf{p}}_v)}{\|\mathbf{p}_v\|}\|\mathbf{p}_r\| \cos \beta \sin \beta + \frac{(\mathbf{p}_r \cdot \Dot{\mathbf{p}}_r)}{\|\mathbf{p}_r\|}\|\mathbf{p}_v\| \cos \beta \sin \beta + \|\mathbf{p}_v\| \|\mathbf{p}_r\| (\cos^2 \beta - \sin^2 \beta) \Dot{\beta} \right] \\
        & - \frac{(\mathbf{r} \cdot \mathbf{v})^2}{\|\mathbf{r}\|}\|\mathbf{p}_r\|^2 - 2\|\mathbf{r}\| (\mathbf{p}_r \cdot \Dot{\mathbf{p}}_r)(\mathbf{r} \cdot \mathbf{v}) - \|\mathbf{r}\| \|\mathbf{p}_r\|^2 [\|\mathbf{v}\|^2 + (\mathbf{r} \cdot \Dot{\mathbf{v}})] \\
        & + 2\mu \left[ 2(\mathbf{p}_v \cdot \Dot{\mathbf{p}}_v) \cos \beta \sin \beta \Dot{\beta} + \|\mathbf{p}_v\|^2 (\cos^2 \beta - \sin^2 \beta) \Dot{\beta}^2 + \|\mathbf{p}_v\|^2 \cos \beta \sin \beta \Ddot{\beta} \right] = 0
\end{split}
\label{D3dot_costates}
\end{equation}
Next consider that
\begin{equation}
\begin{split}
        \mathbf{r} \cdot \Dot{\mathbf{v}} &= \mathbf{r} \cdot \left( -\mu \frac{\mathbf{r}}{\|\mathbf{r}\|^3} + \frac{T_{\text{max}}}{m}c \mathbf{n}\right) \\
        & -\frac{\mu}{\|\mathbf{r}\|} + \frac{T_{\text{max}}}{m} c \|\mathbf{r}\| \cos \beta
\end{split}
\label{rvdot}
\end{equation}
Using Lemma \ref{lemma2} and Eqs.\ \cref{prprdot,alphaDbetaD_mod,rvdot}, the equation $\dot{D}_3 = 0$ becomes
\begin{equation}
\begin{split}
        \dot{D}_3 &= -\frac{6\mu^2}{\|\mathbf{r}\|^3} \cos^2 \beta \sin^2 \beta {\color{blue} \mp} \frac{2\mu^2}{\|\mathbf{r}\|^3}\sqrt{(1 - 3\cos^2 \beta)(1 + e \cos \theta)} (\cos^2 \beta - \sin^2 \beta)  \left(-1 {\color{red} \pm} \sqrt{\frac{1 - 3 \cos^2 \beta}{1 + e \cos \theta}} \right) \\
        & - \frac{\mu}{\|\mathbf{r}\|^2} (1 - 3\cos^2 \beta) {\color{blue}\mp}  \frac{6\mu}{\|\mathbf{r}\|^2} \cos \beta \sin \beta \sqrt{\frac{\mu}{\|\mathbf{r}\|^3} (1 - 3\cos^2 \beta)} (\mathbf{r} \cdot \mathbf{v}) \\
        & - \frac{\mu}{\|\mathbf{r}\|^2} (1 - 3\cos^2 \beta)[\|\mathbf{v}\|^2 -\frac{\mu}{\|\mathbf{r}\|} + \frac{T_{\text{max}}}{m} c \|\mathbf{r}\| \cos \beta] \\
        & + 2\frac{\mu^2}{\|\mathbf{r}\|^3} (\cos^2 \beta - \sin^2 \beta) (1 + e \cos \theta) \left(-1 {\color{red} \pm} \sqrt{\frac{1 - 3 \cos^2 \beta}{1 + e \cos \theta}} \right)^2 + 2\mu \cos \beta \sin \beta \Ddot{\beta} = 0
\end{split}
\label{D3dot_states2}
\end{equation}
Note that in the above equation the control throttle factor $c$ appears. Now, let us analyze the expression of the term $\|\mathbf{v}\|^2$. According to the classical orbital mechanics, the component of the velocity along the radius is expressed in Eq.\ \cref{vr}, whereas the component of the velocity along the direction perpendicular to the radius is \cite{curtis2013orbital}
\begin{equation}
    v_{\perp} = \frac{\mu}{\|\mathbf{h}\|} (1 + e \cos \theta) = \frac{\mu}{\sqrt{\mu \|\mathbf{r}\| (1 + e \cos \theta)}} (1 + e \cos \theta) = \sqrt{\frac{\mu (1 + e \cos \theta)}{\|\mathbf{r}\|}}
\end{equation}
From the above equations we have that
\begin{equation}
    \|\mathbf{v}\|^2 = \frac{\mu}{\|\mathbf{r}\| } \frac{e^2 \sin^2 \theta}{1 + e \cos \theta} + \frac{\mu (1 + e \cos \theta)}{\|\mathbf{r}\|} = \frac{\mu}{\|\mathbf{r}\|} \left[ \frac{(1 + e \cos \theta)^2 + e^2 \sin^2 \theta}{(1 + e \cos \theta)} \right]
    \label{v2}
\end{equation}
By substituting Eqs.\ \cref{rv} and \cref{v2} into Eq.\ \cref{D3dot_states2}, we get
\begin{equation}
\begin{split}
        \dot{D}_3 &= -\frac{6\mu}{\|\mathbf{r}\|^2} \cos^2 \beta \sin^2 \beta {\color{blue} \mp} \frac{2\mu}{\|\mathbf{r}\|^2}\sqrt{(1 - 3\cos^2 \beta)(1 + e \cos \theta)} (\cos^2 \beta - \sin^2 \beta) \left(-1 {\color{red} \pm} \sqrt{\frac{1 - 3 \cos^2 \beta}{1 + e \cos \theta}} \right) \\
        & - \frac{\mu}{\|\mathbf{r}\|^2}\frac{e^2 \sin^2 \theta}{(1 + e \cos \theta)}(1 - 3\cos^2 \beta) {\color{blue}\mp}  \frac{6\mu}{\|\mathbf{r}\|^2} \cos \beta \sin \beta \sqrt{\frac{1 - 3\cos^2 \beta}{1 + e \cos \theta}}(e \sin \theta) \\
        & - \frac{\mu}{\|\mathbf{r}\|^2} (1 - 3\cos^2 \beta)\left[\frac{(1 + e \cos \theta)^2 + e^2 \sin^2 \theta}{(1 + e \cos \theta)} -1 \right] - (1 - 3\cos^2 \beta)\frac{T_{\text{max}}}{m} c \cos \beta \\
        & + 2\frac{\mu}{\|\mathbf{r}\|^2} (\cos^2 \beta - \sin^2 \beta) (1 + e \cos \theta) \left(-1 {\color{red} \pm} \sqrt{\frac{1 - 3 \cos^2 \beta}{1 + e \cos \theta}} \right)^2 + 2\cos \beta \sin \beta \Ddot{\beta} = 0
\end{split}
\label{D3dot_states3}
\end{equation}
Therefore, we have expressed the term $\dot{D}_3$ only as a function of $\|\mathbf{r}\|$, $e$, $\beta$, $\theta$, and $\Ddot{\beta}$. In order to cancel out the dependency on $\Ddot{\beta}$, we proceed as follows. Let us consider Eq.\ \cref{D2} and take its derivative
\begin{equation}
    -3\|\mathbf{r}\| (\mathbf{r} \cdot \mathbf{v})(\Dot{\alpha} + \Dot{\beta})^2 - 2(\Dot{\alpha} + \Dot{\beta})(\Ddot{\alpha} + \Ddot{\beta})\|\mathbf{r}\|^3 +  6\mu \cos \beta \sin \beta \Dot{\beta} = 0
\end{equation}
By substituting the expressions for $\Dot{\beta}$, $(\Dot{\alpha} + \Dot{\beta})$, and $(\mathbf{r} \cdot \mathbf{v})$:
\begin{equation}
\begin{split}
        & -3\sqrt{\frac{\mu^3}{\|\mathbf{r}\|^3}} \frac{e \sin \theta}{\sqrt{1 + e \cos \theta}} (1 - 3\cos^2 \beta) {\color{red} \mp} 2\sqrt{\mu\|\mathbf{r}\|^3 (1 - 3\cos^2 \beta)}(\Ddot{\alpha} + \Ddot{\beta}) \\
        & + 6\mu \cos \beta \sin \beta \sqrt{\frac{\mu}{\|\mathbf{r}\|^3}}\sqrt{1 + e \cos \theta} \left(-1 {\color{red} \pm} \sqrt{\frac{1 - 3 \cos^2 \beta}{1 + e \cos \theta}} \right) = 0
\end{split}
\label{beta2dotfirst}
\end{equation}
From Eq.\ \cref{alphadot},
\begin{equation}
\begin{split}
        \Ddot{\alpha} &= \frac{\text{d}}{\text{d}t} \left( \frac{\|\mathbf{h}\|}{\|\mathbf{r}\|^2} \right) \\ & = \frac{\text{d}}{\text{d}t} \left( \frac{\sqrt{\mu \|\mathbf{r}\| (1 + e \cos \theta)}}{\|\mathbf{r}\|^2} \right) \quad \leftarrow \text{using Eq.\ \cref{h}} \\
        &= \sqrt{\mu} \frac{\text{d}}{\text{d}t} \left( \sqrt{\frac{(1 + e \cos \theta)}{\|\mathbf{r}\|^3}} \right) \\
        & = \sqrt{\mu} \left( -\frac{3}{2} \|\mathbf{r}\|^{-7/2}(\mathbf{r} \cdot \mathbf{v}) \sqrt{1 + e \cos \theta} +  \frac{\Dot{e}\cos \theta - e \Dot{\theta}\sin \theta}{2\sqrt{1 + e \cos \theta}} \|\mathbf{r}\|^{-3/2}\right) \\
        & =  -\frac{3}{2}\frac{\mu}{\|\mathbf{r}\|^{3}} e \sin \theta +  \sqrt{\frac{\mu}{\|\mathbf{r}\|^3}}\frac{\Dot{e}\cos \theta - e \Dot{\theta}\sin \theta}{2\sqrt{1 + e \cos \theta}} \quad \leftarrow \text{using Eq.\ \cref{rv}}\\ 
        & =  -\frac{3}{2}\frac{\mu}{\|\mathbf{r}\|^{3}} e \sin \theta + K(\|\mathbf{r}\|, e, \theta, \Dot{e}, \Dot{\theta})
\end{split}
\label{alpha2dot}
\end{equation}
From the classical orbital mechanics \cite{curtis2013orbital}, we have that the evolution in time of the eccentricity and the true anomaly is described by
\begin{equation}
    \begin{split}
        \Dot{e} & = \sqrt{\frac{\|\mathbf{r}\| (1 + e \cos \theta)}{\mu}}\sin \theta \frac{T_r}{m} + \sqrt{\frac{\|\mathbf{r}\|}{\mu(1 + e \cos \theta)}} \left[(2 + e \cos \theta) \cos \theta + e \right] \frac{T_s}{m}\\
        \Dot{\theta} & =  \sqrt{\frac{\mu}{\|\mathbf{r}\|^3} (1 + e \cos \theta)} + \sqrt{\frac{ \|\mathbf{r}\|}{e \mu (1 + e \cos \theta)}} \left[ (1 + e \cos \theta) \cos \theta \frac{T_r}{m} - \left( 2 +  e \cos \theta\right) \sin \theta \frac{T_s}{m}\right] 
    \end{split}
    \label{elemdot}
\end{equation}
In the above equations, $T_r$ and $T_s$ are the components of the thrust along the radius and its perpendicular direction, respectively. That is,
\begin{equation}
    \begin{split}
        & T_r = cT_{\text{max}} \cos{\beta} \\
        & T_s = cT_{\text{max}} \sin{\beta}
    \end{split}
    \label{TrTs}
\end{equation}
Many works in the literature use the Gauss variational equations or some modified version to represent the equations of motion of a spacecraft equipped with low-thrust engines \cite{hudson2009reduction,junkins2019exploration,gurfil2007nonlinear}. Now, let us analyze the term $K$ in Eq.\ \cref{alpha2dot}. Consider first that $K = K_1 + K_2$, where 
\begin{equation}
    \begin{split}
        & K_1 =  \sqrt{\frac{\mu}{\|\mathbf{r}\|^3}}\frac{\Dot{e}\cos \theta}{2\sqrt{1 + e \cos \theta}} \\
        & K_2 =  -\sqrt{\frac{\mu}{\|\mathbf{r}\|^3}}\frac{e \Dot{\theta}\sin \theta}{2\sqrt{1 + e \cos \theta}}
    \end{split}
    \label{K1K2}
\end{equation}
First consider $K_1$. If we substitute the expression for $\Dot{e}$, we get
\begin{equation}
    \begin{split}
        K_1 &=  \sqrt{\frac{\mu}{\|\mathbf{r}\|^3}}\frac{\Dot{e}\cos \theta}{2\sqrt{1 + e \cos \theta}} \\
        & = \sqrt{\frac{\mu}{\|\mathbf{r}\|^3}}\frac{\cos \theta}{2\sqrt{1 + e \cos \theta}}\left\{ \sqrt{\frac{\|\mathbf{r}\| (1 + e \cos \theta)}{\mu}}\sin \theta \frac{T_r}{m} + \sqrt{\frac{\|\mathbf{r}\|}{\mu(1 + e \cos \theta)}} \left[(2 + e \cos \theta) \cos \theta + e \right] \frac{T_s}{m}
        \right\} \\
        & = \frac{\cos \theta \sin \theta}{2\|\mathbf{r}\|}\frac{T_r}{m} + \frac{\cos^2 \theta (2 + e \cos \theta)}{2\|\mathbf{r}\|(1 + e \cos \theta)}\frac{T_s}{m} + \frac{e \cos \theta}{2\|\mathbf{r}\|(1 + e \cos \theta)}\frac{T_s}{m}  \\
        & = \frac{\cos \theta \sin \theta}{2\|\mathbf{r}\|}c\frac{T_{\text{max}}}{m}\cos{\beta} + \frac{\cos^2 \theta (2 + e \cos \theta)}{2\|\mathbf{r}\|(1 + e \cos \theta)}c\frac{T_{\text{max}} }{m}\sin{\beta} + \frac{e \cos \theta}{2\|\mathbf{r}\|(1 + e \cos \theta)}c\frac{T_{\text{max}}}{m}\sin{\beta} \quad \leftarrow \text{from Eq.\ \cref{TrTs}} \\
        & = c\frac{T_{\text{max}}}{m} \frac{1}{2\|\mathbf{r}\|}\left[ \cos \theta \sin \theta \cos{\beta} + \frac{\cos^2 \theta (2 + e \cos \theta)}{1 + e \cos \theta}\sin{\beta} + \frac{e \cos \theta}{1 + e \cos \theta}\sin{\beta}\right] \\
        & = c\frac{T_{\text{max}}}{m}H_1(\|\mathbf{r}\|,e, \theta, \beta)
    \end{split}
    \label{K1exp}
\end{equation}
We can proceed in the same way for $K_2$:
\begin{equation}
    \begin{split}
        K_2 &= -\sqrt{\frac{\mu}{\|\mathbf{r}\|^3}}\frac{e \Dot{\theta}\sin \theta}{2\sqrt{1 + e \cos \theta}} \\
        & =  -\sqrt{\frac{\mu}{\|\mathbf{r}\|^3}}\frac{e \sin \theta}{2\sqrt{1 + e \cos \theta}} \sqrt{\frac{\mu}{\|\mathbf{r}\|^3} (1 + e \cos \theta)} \\
        & -\sqrt{\frac{\mu}{\|\mathbf{r}\|^3}}\frac{e \sin \theta}{2\sqrt{1 + e \cos \theta}} \sqrt{\frac{ \|\mathbf{r}\|}{e \mu (1 + e \cos \theta)}}\left[(1 + e \cos \theta) \cos \theta \frac{T_r}{m} - \left( 2 +  e \cos \theta\right) \sin \theta \frac{T_s}{m}\right] \\
        & = -\frac{\mu}{\|\mathbf{r}\|^3}\frac{e \sin \theta}{2} - \frac{\sin \theta \cos \theta}{2\|\mathbf{r}\|} \frac{T_r}{m} + \frac{\sin^2 \theta (2 + e \cos \theta)}{2\|\mathbf{r}\|(1 + e \cos \theta)}\frac{T_s}{m} \quad \leftarrow \text{from Eq.\ \cref{TrTs}} \\
        & = -\frac{\mu}{\|\mathbf{r}\|^3}\frac{e \sin \theta}{2} + c\frac{T_{\text{max}}}{m}\frac{1}{2\|\mathbf{r}\|} \left[- \cos \theta \sin \theta \cos \beta + \frac{\sin^2 \theta (2 + e \cos \theta)}{1 + e \cos \theta} \sin \beta \right] \\
        & = -\frac{\mu}{\|\mathbf{r}\|^3}\frac{e \sin \theta}{2} + c\frac{T_{\text{max}}}{m} H_2(\|\mathbf{r}\|, e, \theta, \beta)
    \end{split}
    \label{K2exp}
\end{equation}
Therefore, using Eqs.\ \cref{K1K2,K1exp,K2exp}, the expression of $\Ddot{\alpha}$ in Eq.\ \cref{alpha2dot} becomes:
\begin{equation}
\begin{split}
        \Ddot{\alpha} &= -\frac{3}{2}\frac{\mu}{\|\mathbf{r}\|^{3}} e \sin \theta + c\frac{T_{\text{max}}}{m}H_1(\|\mathbf{r}\|,e, \theta, \beta)-\frac{\mu}{\|\mathbf{r}\|^3}\frac{e \sin \theta}{2} + c\frac{T_{\text{max}}}{m} H_2(\|\mathbf{r}\|, e, \theta, \beta) \\
        &= -2 \frac{\mu}{\|\mathbf{r}\|^{3}} e \sin \theta + c\frac{T_{\text{max}}}{m} \left[ H_1(\|\mathbf{r}\|,e, \theta, \beta) +  H_2(\|\mathbf{r}\|, e, \theta, \beta) \right]
\end{split}
\label{alpha2dot2}
\end{equation}
Let us analyze the term $H_1(\|\mathbf{r}\|,e, \theta, \beta) +  H_2(\|\mathbf{r}\|, e, \theta, \beta)$:
\begin{equation}
    \begin{split}
        H_1(\|\mathbf{r}\|,e, \theta, \beta) +  H_2(\|\mathbf{r}\|, e, \theta, \beta) & =  \frac{1}{2\|\mathbf{r}\|} \left[\cos \theta \sin \theta \cos{\beta} + \frac{\cos^2 \theta (2 + e \cos \theta)}{1 + e \cos \theta}\sin{\beta} + \frac{e \cos \theta}{1 + e \cos \theta}\sin{\beta}\right] \\
        & + \frac{1}{2\|\mathbf{r}\|} \left[- \cos \theta \sin \theta \cos \beta + \frac{\sin^2 \theta (2 + e \cos \theta)}{1 + e \cos \theta} \sin \beta \right] \\
        & =  \frac{1}{2\|\mathbf{r}\|}\frac{\cos^2 \theta (2 + e \cos \theta) + e \cos \theta + \sin^2 \theta (2 + e \cos \theta)}{1 + e \cos \theta} \sin \beta \\
        & =  \frac{1}{2\|\mathbf{r}\|}\frac{\cos^2 \theta (2 + e \cos \theta) + e \cos \theta + (1 - \cos^2 \theta) (2 + e \cos \theta)}{1 + e \cos \theta} \sin \beta \\
        & =  \frac{1}{2\|\mathbf{r}\|}2 \sin \beta = \frac{\sin \beta}{\|\mathbf{r}\|}
    \end{split}
\end{equation}
Eq.\ \cref{alpha2dot2} reduces therefore to
\begin{equation}
    \Ddot{\alpha} = -2 \frac{\mu}{\|\mathbf{r}\|^{3}} e \sin \theta + c\frac{T_{\text{max}}}{m} \frac{\sin \beta}{\|\mathbf{r}\|}
    \label{alpha2dot3}
\end{equation}
Using Eq.\ \cref{alpha2dot3}, from Eq.\ \cref{beta2dotfirst} we can find
\begin{equation}
\begin{split}
         \Ddot{\beta} &= {\color{red} \mp}\frac{-6 \cos \beta \sin \beta \sqrt{\frac{\mu^3}{\|\mathbf{r}\|^3}(1 + e \cos \theta)} \left(-1 {\color{red} \pm}  \sqrt{\frac{1 - 3 \cos^2 \beta}{1 + e \cos \theta}} \right) + 3\sqrt{\frac{\mu^3}{ \|\mathbf{r}\|^3}} \frac{e \sin \theta}{\sqrt{1 + e \cos \theta}} (1 - 3\cos^2 \beta)}{ 2\sqrt{\mu \|\mathbf{r}\|^3 (1 - 3\cos^2 \beta)}} \\
         & +2 \frac{\mu}{\|\mathbf{r}\|^{3}} e \sin \theta - c\frac{T_{\text{max}}}{m} \frac{\sin \beta}{\|\mathbf{r}\|}\\
         & = D(\|\mathbf{r}\|, e, \theta, \beta) - c\frac{T_{\text{max}}}{m} \frac{\sin \beta}{\|\mathbf{r}\|}
\end{split}
\end{equation}
Finally, Eq.\ \cref{D3dot_states3} can be rewritten by substituting the expression of $\Ddot{\beta}$ as
\begin{equation}
\begin{split}
        \dot{D}_3 &= -\frac{6\mu}{\|\mathbf{r}\|^3} \cos^2 \beta \sin^2 \beta {\color{blue} \mp} \frac{2\mu}{\|\mathbf{r}\|^3}\sqrt{1 - 3\cos^2 \beta} (\cos^2 \beta - \sin^2 \beta) \sqrt{1 + e \cos \theta} \left(-1 {\color{red} \pm} \sqrt{\frac{1 - 3 \cos^2 \beta}{1 + e \cos \theta}} \right) \\
        & - \frac{\mu}{\|\mathbf{r}\|^3}\frac{e^2 \sin^2 \theta}{(1 + e \cos \theta)}(1 - 3\cos^2 \beta) {\color{blue}\mp}  \frac{6\mu}{\|\mathbf{r}\|^3} \cos \beta \sin \beta \sqrt{1 - 3\cos^2 \beta}\frac{e \sin \theta}{\sqrt{1 + e \cos \theta}} \\
        & - \frac{\mu}{\|\mathbf{r}\|^3} (1 - 3\cos^2 \beta)\left[\frac{(1 + e \cos \theta)^2 + e^2 \sin^2 \theta}{(1 + e \cos \theta)} -1 \right] - \frac{1}{\|\mathbf{r}\|} (1 - 3\cos^2 \beta)\frac{T_{\text{max}}}{m} c \cos \beta \\
        & + 2\frac{\mu}{\|\mathbf{r}\|^3} (\cos^2 \beta - \sin^2 \beta) (1 + e \cos \theta) \left(-1 {\color{red} \pm} \sqrt{\frac{1 - 3 \cos^2 \beta}{1 + e \cos \theta}} \right)^2 + 2\cos \beta \sin \beta \left( D(\|\mathbf{r}\|, e, \theta, \beta) - c\frac{T_{\text{max}}}{m} \frac{\sin \beta}{\|\mathbf{r}\|}\right) = 0
\end{split}
\label{D3dot_states4}
\end{equation}
The above equation only depends on $\|\mathbf{r}\|, e, m, \theta, \beta$, and $c$. Therefore, we can write
\begin{equation}
    c_s = \frac{B(\|\mathbf{r}\|, e, m, \theta, \beta)}{A(\beta)}
    \label{singcon}
\end{equation}
where
\begin{equation}
\begin{split}
        A &= (1 - 3\cos^2 \beta) \cos \beta + 2 \cos \beta \sin^2 \beta
\end{split}
\end{equation}
and
\begin{equation}
\begin{split}
        B &= \|\mathbf{r}\|\frac{m} {T_{\text{max}}} \bigg \{-\frac{6\mu}{\|\mathbf{r}\|^3} \cos^2 \beta \sin^2 \beta {\color{blue} \mp} \frac{2\mu}{\|\mathbf{r}\|^3}\sqrt{1 - 3\cos^2 \beta} (\cos^2 \beta - \sin^2 \beta) \sqrt{1 + e \cos \theta} \left(-1 {\color{red} \pm} \sqrt{\frac{1 - 3 \cos^2 \beta}{1 + e \cos \theta}} \right) \\
        & - \frac{\mu}{\|\mathbf{r}\|^3}\frac{e^2 \sin^2 \theta}{(1 + e \cos \theta)}(1 - 3\cos^2 \beta) {\color{blue}\mp}  \frac{6\mu}{\|\mathbf{r}\|^3} \cos \beta \sin \beta \sqrt{1 - 3\cos^2 \beta}\frac{e \sin \theta}{\sqrt{1 + e \cos \theta}} \\
        & - \frac{\mu}{\|\mathbf{r}\|^3} (1 - 3\cos^2 \beta)\left[\frac{(1 + e \cos \theta)^2 + e^2 \sin^2 \theta}{(1 + e \cos \theta)} -1 \right] \\
        & + 2\frac{\mu}{\|\mathbf{r}\|^3} (\cos^2 \beta - \sin^2 \beta) (1 + e \cos \theta) \left(-1 {\color{red} \pm} \sqrt{\frac{1 - 3 \cos^2 \beta}{1 + e \cos \theta}} \right)^2 + 2\cos \beta \sin \beta  D(\|\mathbf{r}\|, e, \theta, \beta) \bigg \}
\end{split}
\label{Bterm}
\end{equation}
\end{proof}
\subsection{Proof of Theorem 2.2}
\begin{proof}
For $A(\beta)$ to be zero it must be
\begin{equation}
\begin{split}
        (1 - 3\cos^2 \beta) \cos \beta + 2 \cos \beta \sin^2 \beta &= 0 \\
        \cos \beta (1 - 3 \cos^2 \beta + 2\sin^2 \beta) & = 0 \\
        \cos \beta (1 - 3 +3\sin^2 \beta + 2\sin^2 \beta) & = 0 \\
        \cos \beta (-2 + 5 \sin^2 \beta) & = 0
\end{split}
\end{equation}
Which means
\begin{equation}
    \cos \beta = 0 \vee \sin \beta = \pm \sqrt{\frac{2}{5}}
\end{equation}
However, the second case is not compatible with the condition in Corollary \ref{corollary1}. Therefore, the only case for which $A = 0$ happens when $\beta = \frac{\pi}{2} + k\pi$, $k \in \mathbb{Z}$. According to Eq.\ \cref{algebraic}, this corresponds to the following cases: $e = 0$ or $\sin \theta = 0$, hence the first point of Theorem 2.1. To demonstrate the second point, consider that if $\beta = \frac{\pi}{2} + k\pi$, $k \in \mathbb{Z}$, then $T_r = 0$ $T_s = \pm cT_{\text{max}}$. Consequently, Eqs.\ \cref{elemdot} become
\begin{equation}
    \begin{split}
        \Dot{e} & =  \sqrt{\frac{\|\mathbf{r}\|}{\mu(1 + e \cos \theta)}} \left[(2 + e \cos \theta) \cos \theta + e \right] \frac{T_s}{m}\\
        \Dot{\theta} & =  \sqrt{\frac{\mu}{\|\mathbf{r}\|^3} (1 + e \cos \theta)} + \sqrt{\frac{ \|\mathbf{r}\|}{e \mu (1 + e \cos \theta)}} \left[ - \left( 2 +  e \cos \theta\right) \sin \theta \frac{T_s}{m}\right] 
    \end{split}
    \label{elemdotmod}
\end{equation}
Now, if the case $e = 0$ is considered, 
\begin{equation}
    \begin{split}
        \Dot{e} & =  \pm2\sqrt{\frac{\|\mathbf{r}\|}{\mu}} \cos \theta  \frac{cT_{\text{max}}}{m}
    \end{split}
\end{equation}
which can only be zero if either $c = 0$ (and thus in a non-singular arc case) or when $\cos \theta = 0$. Consider now the case in which $\sin \theta = 0$.  From Eq.\ \cref{elemdotmod}, 
\begin{equation}
    \begin{split}
        \Dot{\theta} & =  \sqrt{\frac{\mu}{\|\mathbf{r}\|^3} (1 + e \cos \theta)} > 0
    \end{split}
    \label{elemdotmod}
\end{equation}

\end{proof}
\subsection{Summary of the Theoretical Results}
The theoretical results obtained in Section \ref{sec:strategy} can be summarized as follows.
\begin{enumerate}
    \item We have found the algebraic necessary condition expressed in Eq.\ \cref{algebraici} to have singular arcs that depends on three physical variables only, namely the eccentricity of the spacecraft, its true anomaly, and the angle $\beta$. Previous works in literature have only found necessary conditions that depend on the state \textit{and} the costates \cite{park2013necessary}. This represents a major improvement because an easier evaluation of the necessary condition can be performed. Moreover, the evaluation of the algebraic condition provides a physical grasp on the problem.
    \item In Corollary \ref{corollary2}, we have shown that the solutions of the algebraic necessary conditions for each fixed couple of eccentricity and true anomaly are, in number, between six and ten. This means that the thruster must assume specific directions, hence suggesting the reasons of the rareness of singular arcs.
    \item Leveraging the Gauss variational equations, we have expressed the singular throttle factor as the ratio of two algebraic expressions that solely depend on a limited set of physical variables. The expression of the singular throttle factor is provided in Theorem \ref{theorem2}. As per the previous point, past works in literature were only able to express it using the costates.
\end{enumerate}
So far, reference works in the field of low-thrust trajectory optimization have assumed singular arcs are negligible when computing interplanetary trajectories, although a complete theoretical framework to justify this assumption was still missing.
\section{Numerical Simulations}
\label{sec:genass}
\subsection{Algebraic Necessary Conditions}
This section shows the solutions of the algebraic condition in Eq.\ \cref{algebraic}. Taking into account the fact that two signs can vary, the equation actually represents four conditions. However, due to the symmetry of the condition itself, they reduce to three. Since the equation is not solvable in closed form, the solutions are obtained numerically with the MATLAB\textregistered \, function \textit{fzero} and the default solver, which requires an initial guess. We relied on Corollary \ref{corollary2} to provide the correct number of initial guesses for each couple $(e, \theta)$ such that all the solutions of the equation were found. Figure \ref{fig:algeb1} shows the angle $\beta$ that respects the condition for the case $\sin \beta > 0$. 
\begin{figure}[!ht]
    \centering
    {\includegraphics[width=\textwidth]{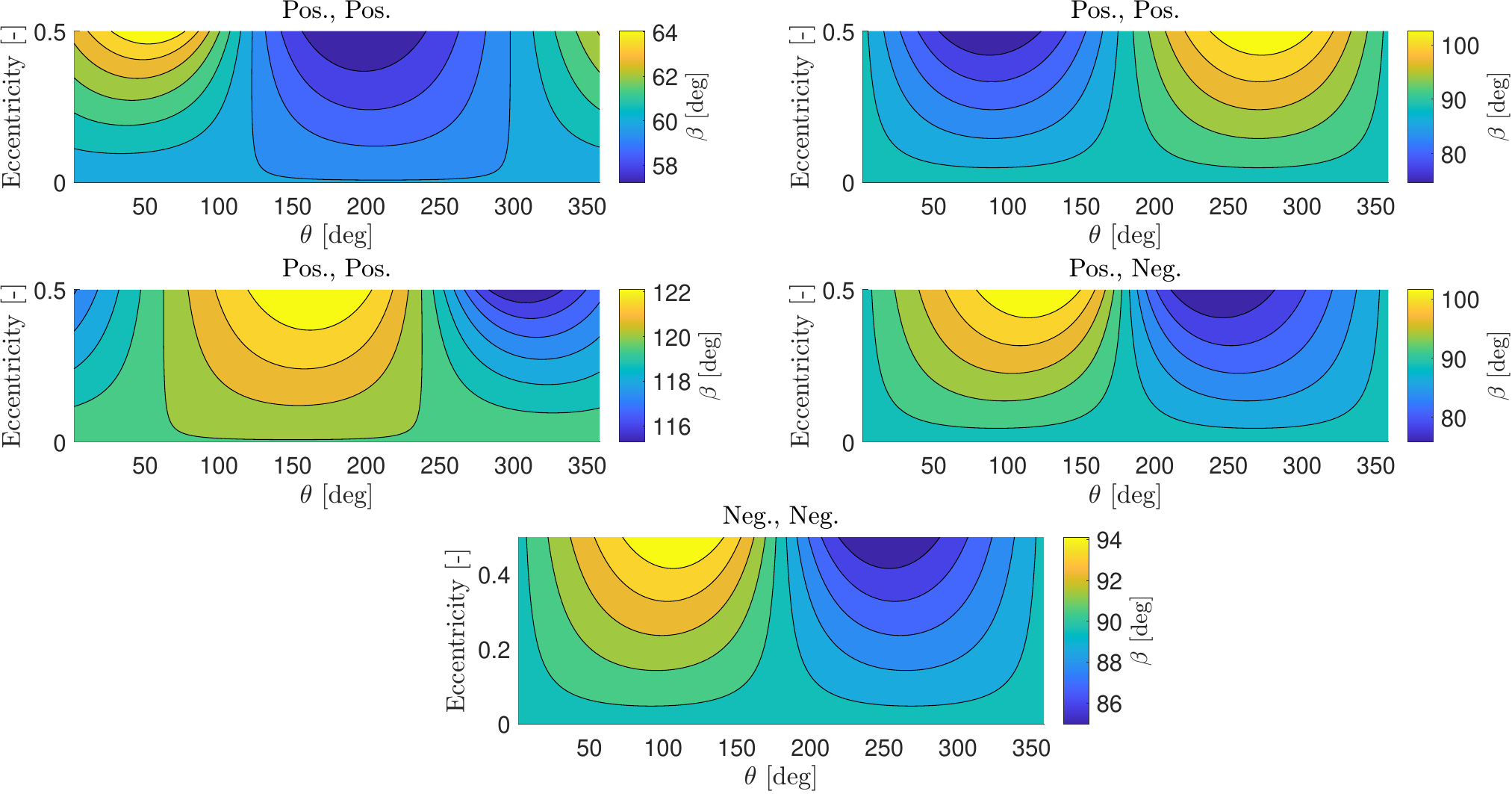}}
    \caption{Solutions of the equation $\Psi(e, \theta, \beta) = 0$ when $\sin \beta > 0$.}
\label{fig:algeb1}
\end{figure}
Figure \ref{fig:algeb2} shows the same as the previous figures, respectively, for the case $\sin \beta < 0$. Note that, due to the symmetry of the necessary conditions, the values $\beta_2$ correspondent to the cases $\sin \beta < 0$ are such that $\beta_2 = \beta_1 + \pi$, where $\beta_1$ are the angles that correspond to the case $\sin \beta > 0$. In the plots, the eccentricity varies between the values $0$ and $0.5$, as our simulations on several low-thrust trajectories show that it rarely overcomes the value of $0.5$. Note that for the selected intervals of $e$ and $\theta$, the algebraic necessary condition has always ten zeros. If cases with $e > 0.5$ are considered, it may happen that it has less, as Fig.\ \ref{examplealg} shows.
\begin{figure}[!ht]
    \centering
    {\includegraphics[width=\textwidth]{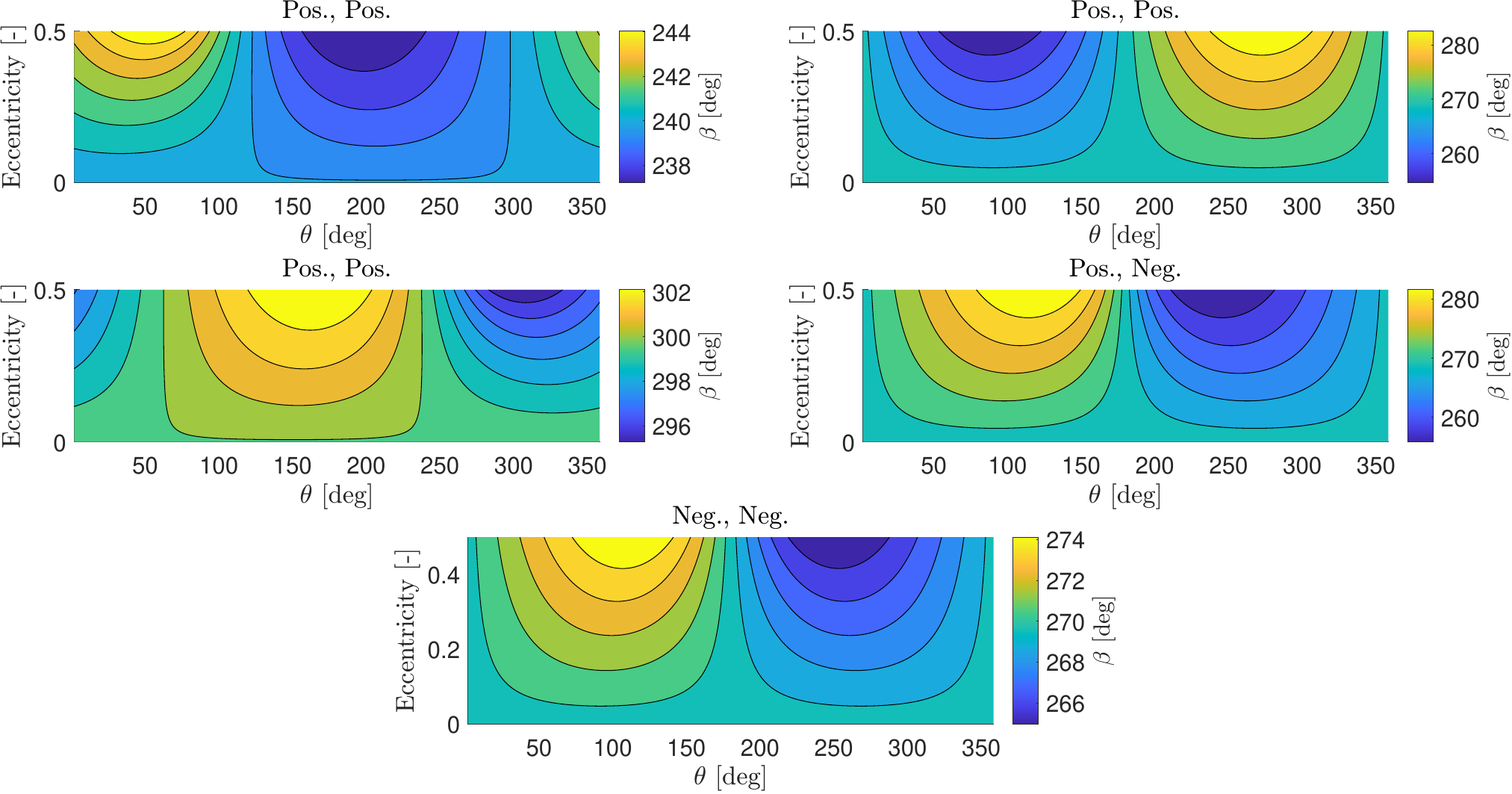}}
    \caption{Solutions of the equation $\Psi(e, \theta, \beta) = 0$ when $\sin \beta < 0$.}
\label{fig:algeb2}
\end{figure}
\begin{figure*}[h]
	\centering
\includegraphics[width=\textwidth]{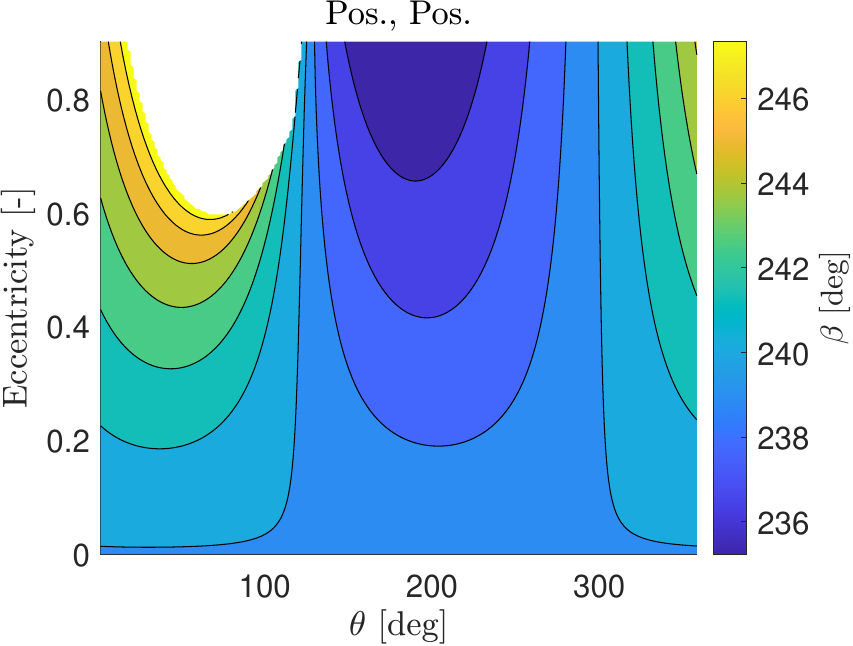}
	\caption{Example cases for which the algebraic necessary condition has less than ten zeros.}
	\label{examplealg}
\end{figure*}
\subsection{Value of the Singular Control}
In this section, we use Eq.\ \cref{singcon} to find the value of the singular control for a grid of parameters $\|\mathbf{r}\|, e, m, \theta$. We consider:
\begin{enumerate}
    \item a uniform grid of radii from $\|\mathbf{r}_{\text{min}}\| = 0.1$ to $\|\mathbf{r}_{\text{max}}\| = 15$;
    \item eccentricity values ranging from $e_{\text{min}} = 10^{-3}$ to $e_{\text{max}} = 0.9$; the minimum eccentricity value is not zero to avoid the singularity of the term $A(\beta)$ presented in Theorem 2.2.
    \item a uniform grid of true anomaly values from $\theta_{\text{min}} = 10^{-2}$ to $\theta_{\text{max}} = 1.99\pi$; the minimum and the maximum true anomaly values are not zero and $2\pi$, respectively, to avoid the singularity of the term $A(\beta)$ presented in Theorem 2.2.
    \item a mass of value $1$;
\end{enumerate}
Note that all the above values are dimensionless. The value of $\beta$ to be used comes from the algebraic necessary condition given a couple of parameters $(e, \theta)$. For all combinations of parameters, we evaluate the singular control $c_s$ in Eq.\ \cref{singcon}. If $c_s > 1 - \varepsilon$ or $c_s < \varepsilon$, with $\varepsilon = 10^{-3}$, this means that the controls are in fact non-singular. Figure \ref{fig:percentage} shows the overall probability of encountering singular controls as a function of the radius $\| \mathbf{r}\|$ and of eccentricity. The current low-thrust missions are usually sent to Mars, Venus, or the Main Asteroid Belt. Although further destinations are possible, most of low-thrust missions use solar arrays to power engines, and the available power decreases with the inverse of the square of the distance from the Sun. Our simulations show that typical eccentricity values for low-thrust interplanetary missions are in the range of $0.1-0.5$. In these ranges, the overall percentage of possible singular controls is below $10\%$. Note that this does not mean that singular arcs happen with this frequency, but that \textit{in case} the necessary algebraic condition is verified, there is a probability of at most $10\%$ to have singular values. Table \ref{tab:CBdata} shows the distance to the Sun and the inclination to the ecliptic of the major Solar System celestial bodies. Although we developed our work using the assumption of planar dynamics, the inclination to the ecliptic of the majorities of the celestial bodies is  low, therefore it is likely that the results could still be applied to real missions. 
Finally, when the radius reaches the value of approximately $15$, Fig.\ \ref{fig:percentage} shows that if the necessary condition to have singular arcs is satisfied, then the control is indeed singular in $100\%$ of the cases because, for those values, $0 \leq c_s \leq 1$.
\begin{table*}
  \centering
  \setlength\abovecaptionskip{6pt}
  \captionof{table}{Solar System celestial bodies data.}
  \begin{tabular}{lcc} \toprule \toprule 
		Celestial body & Semi-major axis [AU] & Inclination to the ecliptic [deg] \\ \midrule
	Mercury & 0.39 & 7.00 \\
        Venus &  0.72 & 3.39 \\
        Earth & 1.00 & 0.00\\
        NEOs & < 1.30 & Variable \\
        Mars & 1.52 & 1.85 \\
        Main belt & 2.20-3.20 & Variable \\
        Ceres & 2.77 & 10.60 \\
        Jupiter & 5.20 & 1.30 \\
        Saturn & 9.50 & 2.49 \\
        Uranus & 19.20 & 0.77 \\
        Neptune & 30.10 & 1.77\\
        \bottomrule \bottomrule 
	\end{tabular}
  \label{tab:CBdata}
\end{table*}

\begin{figure*}[h]
	\centering
	\includegraphics[width=\textwidth]{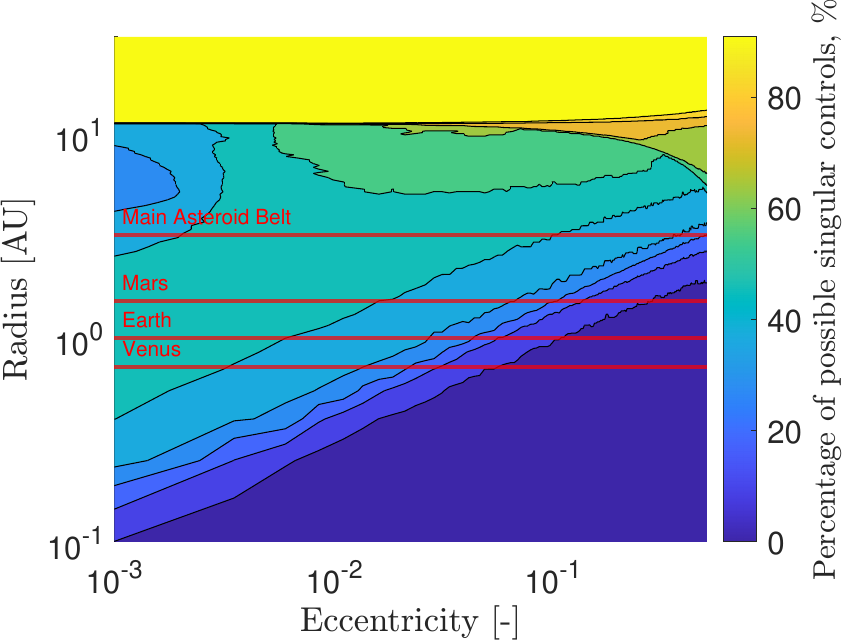}
	\caption{Percentage of possible singular controls over total as a function of the orbit radius and the eccentricity, provided that the algebraic necessary condition is satisfied.}
	\label{fig:percentage}
\end{figure*}

\section{Conclusions}
\label{sec:conclusion}
In this work, algebraic necessary conditions to have singular arcs for the planar two-body low-thrust trajectory optimization problem were presented. The approach, which exploits Gauss variational equations, has allowed to express the necessary conditions as a function of three physical variables only. An analytical expression on the singular control has also been found, which only depends on a limited set of physical variables too. We have shown that the necessary condition is only satisfied if the angle between the thrust direction and the spacecraft radius assumes discrete values, which are in number between six and ten. This suggests the rarity of singular arcs. Moreover, through numerical simulations, we have shown that singular arcs can indeed happen but the cases when the associated throttle factor is singular are relatively rare when trajectories in the inner Solar System need to be designed. Although our approach is specific for planar cases, it is likely that the results can also be extended to three-dimensional cases in practical applications due to the small inclination of most Solar System's celestial bodies. Further work will consist of theoretical investigation of such cases.

\section*{Acknowledgment}

A.\ C.\ M., C.\ G., and F.\ T.\ acknowledge EXTREMA, a project that has received funding from the European Research Council (ERC) under the European Union’s Horizon 2020 research and innovation programme (Grant Agreement No.\ 864697). R.B. acknowledges the French National Research Agency who provided support with funding ANR-22-CE46-0006.

\bibliography{references}

\end{document}